\documentclass[reqno]{amsart}
\usepackage{amsmath}
\usepackage{amsfonts}
\usepackage{amssymb,enumerate}
\usepackage{amsthm}
\usepackage[all]{xy}
\usepackage{rotating}
\usepackage{hyperref}

\theoremstyle{plain}
\newtheorem{lem}{Lemma}[section]
\newtheorem{cor}[lem]{Corollary}
\newtheorem{prop}[lem]{Proposition}
\newtheorem{thm}[lem]{Theorem}
\newtheorem*{theorem*}{Theorem}

\newtheorem*{intthm}{Main Theorem}

\theoremstyle{definition}
\newtheorem{defn}[lem]{Definition}
\newtheorem{ex}[lem]{Example}

\newtheorem{disc}[lem]{Remark}

\newtheorem{notn}[lem]{Notation}
\newtheorem{fact}[lem]{Fact}

\newtheorem*{convention}{Convention}

\newcommand{\catd}{\mathcal{D}}

\newcommand{\lotimes}{\otimes^{\mathbf{L}}}
\newcommand{\HH}{\operatorname{H}}

\newcommand{\s}{\mathfrak{S}}

\newcommand{\im}{\operatorname{Im}}
\newcommand{\shift}{\mathsf{\Sigma}}

\newcommand{\ideal}[1]{\mathfrak{#1}}
\newcommand{\m}{\ideal{m}}

\newcommand{\bbz}{\mathbb{Z}}

\newcommand{\xra}{\xrightarrow}
\newcommand{\xla}{\xleftarrow}

\renewcommand{\geq}{\geqslant}
\renewcommand{\leq}{\leqslant}

\newcommand{\Ext}[4][R]{\operatorname{Ext}_{#1}^{#2}(#3,#4)}	
\newcommand{\Rhom}[3][R]{\mathbf{R}\!\operatorname{Hom}_{#1}(#2,#3)}	

\newcommand{\Otimes}[3][R]{#2\otimes_{#1}#3}
\newcommand{\Hom}[3][R]{\operatorname{Hom}_{#1}(#2,#3)}	
\newcommand{\Tor}[4][R]{\operatorname{Tor}^{#1}_{#2}(#3,#4)}

\newcommand{\und}[1]{#1^{\natural}}
\newcommand{\mult}[2]{\mu^{#1,#2}}
\newcommand{\HomA}[2]{\operatorname{Hom}_{A}(#1,#2)}

\newcommand{\ul}{\underline}

\numberwithin{equation}{lem}

\begin{document}

\bibliographystyle{amsplain}

\title{Liftings and Quasi-Liftings of DG modules}

\dedicatory{To Yuji Yoshino}

\author{Saeed Nasseh}
\author{Sean Sather-Wagstaff}

%\address{Department of Mathematics, North Dakota State University, Dept. $\#$ 2750,
%P. O. Box 6050, Fargo, ND 58108-6050, USA.}

%\email{saeed.nasseh@ndsu.edu}

\thanks{Sather-Wagstaff  was supported in part by North Dakota EPSCoR, 
National Science Foundation Grant EPS-0814442,
and  a grant from the NSA}

\address{Department of Mathematics,
North Dakota State University Dept \# 2750,
PO Box 6050,
Fargo, ND 58108-6050
USA}

\email{saeed.nasseh@ndsu.edu}
\urladdr{https://www.ndsu.edu/pubweb/~{}nasseh/}

\email{sean.sather-wagstaff@ndsu.edu}
\urladdr{http://www.ndsu.edu/pubweb/\~{}ssatherw/}

%\date{\today}

%\dedicatory{}


\begin{abstract}
We prove  lifting results for DG modules that are akin to 
Auslander, Ding, and Solberg's famous lifting results for modules.
\end{abstract}

\maketitle

%\tableofcontents

\section*{Introduction} \label{sec0}

\begin{convention}\label{conv110211a}
Throughout this paper, let $R$ be a commutative noetherian
ring.
\end{convention}

Hochster famously wrote that ``life is really worth living'' in a Cohen-Macaulay 
ring~\cite{hochster:safc}.\footnote{We know of this quote from~\cite{bruns:cmr}.}
For instance, if $R$ is  Cohen-Macaulay and local with maximal regular
sequence $\ul t$, then $R/(\ul t)$ is artinian and the natural
epimorphism $R\to R/(\ul t)$ is  nice enough  to allow for transfer of
properties between the two rings. Thus, if one can prove a result for artinian
local rings, then one can (often) prove a similar result for Cohen-Macaulay local
rings by showing that the desired conclusion descends from $R/(\ul t)$ to $R$.
When $R$ is complete, then this is aided sometimes by the
lifting result of Auslander, Ding, and 
Solberg~\cite[Propositions 1.7 and 2.6]{auslander:lawlom}.

\begin{theorem*}
Let  $\ul t\in R$ be an 
$R$-regular sequence, and let $M$ be a finitely generated
$R/(\ul t)$-module. Assume that $R$ is local and $(\ul t)$-adically complete.
\begin{enumerate}[\rm(a)]
\item If $\Ext[R/(\ul t)] 2MM=0$, then $M$ is ``liftable'' to $R$, that is, 
there is a finitely generated $R$-module $N$ such that
$\Otimes{R/(\ul t)}{N}\cong M$ and $\Tor i{R/(\ul t)}{N}=0$ for all $i\geq 1$.
\item If $\Ext[R/(\ul t)] 1MM=0$, then $M$ has at most one lift to $R$.
\end{enumerate}
\end{theorem*}

In this paper, we are concerned with what happens when 
the sequence $\ul t$ is not $R$-regular. One would like a similar mechanism for reducing
questions about arbitrary local rings to the artinian case. 

It is well known that the map $R\to R/(\ul t)$ is not nice enough in general
to guarantee good descent/lifting behavior. Our perspective\footnote{This
perspective is not original to our work. We learned of it from
Avramov and Iyengar.}
in this matter is
that this is not the right map to consider in general:
the correct one is the natural map from $R$ to the Koszul complex $K=K^R(\ul t)$.
This perspective requires one to make some adjustments.
For instance, $K$ is a differential graded $R$-algebra, so not a commutative
ring in the traditional sense.
This may cause some consternation, but the payoff
can be handsome. For instance, in~\cite{nasseh:lrfsdc}
we use this perspective to answer a 
question of Vasconcelos~\cite{vasconcelos:dtmc}. One of the tools for 
the proof of this result is the following version of Auslander, Ding,
and Solberg's lifting result. Note that we do not assume that
$R$ is local in part~\eqref{item120131a} of this result.

\begin{intthm}
Let $\underline{t}=t_1,\cdots,t_n$ be a sequence of elements of $R$, and
assume that $R$ is $\underline tR$-adically  complete.
Let $D$ be a
DG $K^R(\ul t)$-module that is homologically bounded below and homologically degreewise finite. 
\begin{enumerate}[\rm(a)]
\item\label{item120131a}
If $\Ext[K^R(\ul t)]{2}DD=0$,
then $D$ is quasi-liftable to $R$, that is,
there is a semi-free  $R$-complex $D'$ such that
$D\simeq K^R(\ul t)\otimes_RD'$.
\item\label{item120131b}
Assume that $R$ is local.
If $D$ is quasi-liftable
to $R$ and $\Ext[K^R(\ul t)]{1}DD=0$, then 
any two 
homologically degreewise finite quasi-lifts of $D$ to $R$ are quasiisomorphic
over $R$.
\end{enumerate}
\end{intthm}

This result is proved in Corollaries~\ref{cor110513a} and~\ref{cor110516a},
which
follow from more general results on liftings along morphisms of DG algebras.
Note that it is similar to, but quite different from, some results of Yoshino~\cite{yoshino}.

We briefly describe the contents of the paper.
Section~\ref{sec110211a} contains some background material on 
DG algebras and DG modules.
Section~\ref{DG structures}
contains some structural results for DG modules and homomorphisms between
them.
Finally, Section~\ref{lifting}
is where we prove our Main Theorem.

\section{DG Modules}
\label{sec110211a}

We assume that the reader is  familiar with the category of $R$-complexes.
For clarity, we include a few definitions.

\begin{defn}
  \label{cx}
In this paper, complexes of $R$-modules (``$R$-complexes'' for short) are indexed homologically:
\begin{equation*}
M = \cdots \xra{\partial_{n+2}^{M}} M_{n+1} \xra{\partial_{n+1}^{M}} M_n
\xra{\partial_{n}^{M}} M_{n-1} \xra{\partial_{n-1}^{M}} \cdots.
\end{equation*}
The degree of an element $m\in M$ is denoted $|m|$.
The \emph{tensor product} of two $R$-complexes $M,N$ is denoted $\Otimes MN$,
and the \emph{Hom complex} is denoted $\Hom MN$.
A \emph{chain map} $M\to N$ is a cycle of degree 0 in $\Hom MN$.
An $R$-complex $M$ is 
\emph{homologically bounded below} if $\HH_i(M)=0$ for $i\gg 0$; it is 
\emph{bounded below} if $M_i=0$ for $i\gg 0$.
\end{defn}

Next, we begin our background material on DG algebras; see~\cite{apassov:hddgr,avramov:ifr,avramov:dgha}.

\begin{defn}
  \label{DGK}
  A \emph{commutative differential graded $R$-algebra} (\emph{DG $R$-algebra} for short)
  is an $R$-complex $A$ equipped with a
  chain map
  $\mu^A\colon A\otimes_RA\to A$ with $ab:=\mu^A(a\otimes b)$
  that is:
  \begin{description}
  \item[associative] for all $a,b,c\in A$ we have $(ab)c=a(bc)$;
  \item[unital] there is an element $1\in A_0$ such that for all $a\in A$ we have $1a=a$;
  \item[graded commutative] for all $a,b\in A$ we have $ab = (-1)^{|a||b|}ba$ and $a^2=0$ when
      $|a|$ is odd; and
  \item[positively graded] $A_i=0$ for $i<0$.
\end{description}
The map $\mu^A$ is the \emph{product} on $A$.
Given a DG $R$-algebra $A$, the \emph{underlying algebra} is the
graded commutative  $R$-algebra
$\und{A}=\oplus_{i=0}^\infty A_i$.

A \emph{morphism} of DG $R$-algebras is a chain map
$f\colon A\to B$ between DG $R$-algebras  respecting products and multiplicative identities:
$f(aa')=f(a)f(a')$ and $f(1)=1$.
\end{defn}

\begin{ex}\label{ex110216a}
The ring $R$, considered as a complex concentrated in degree 0, is a DG $R$-algebra.
Given a DG $R$-algebra $A$, the map $R\to A$ given by $r\mapsto r\cdot 1$ is a morphism of DG $R$-algebras.
\end{ex}

\begin{fact}\label{fact110216a}
Let $A$ be a DG $R$-algebra.
The fact that the product on $A$ is a chain map
says that $\partial^A$ satisfies the \emph{Leibniz rule}:
$$
\partial^A_{|a|+|b|}(ab)=\partial^A_{|a|}(a)b+(-1)^{|a|}a\partial^A_{|b|}(b).
$$
It is straightforward to show that the $R$-module $A_0$ is an $R$-algebra.
Moreover, the natural map $A_0\to A$ is a morphism of DG $R$-algebras.
The condition $A_{-1}=0$ implies that $A_0$ surjects onto $\HH_0(A)$
and that $\HH_0(A)$ is an $A_0$-algebra. Furthermore,
the $R$-module $A_i$ is an $A_0$-module, and $\HH_i(A)$ is an $\HH_0(A)$-module
for each $i$.

Given a second DG $R$-algebra $K$, the tensor product
$\Otimes KA$ is also a DG $R$-algebra with multiplication
$(x\otimes a)(x'\otimes a'):=(-1)^{|a||x'|}(xx')\otimes(aa')$.
\end{fact}

\begin{defn}\label{defn110216a}
Let $A$ be a DG $R$-algebra. We say that $A$ is \emph{noetherian}
if $\HH_0(A)$ is noetherian and the $\HH_0(A)$-module $\HH_i(A)$ is  finitely generated for all $i\geq 0$.
When $(R,\m)$ is local, we say that $A$ is  \emph{local} if it is noetherian and
the ring $\HH_0(A)$ is a local $R$-algebra\footnote{This means that
$\HH_0(A)$ is a local  ring whose maximal ideal contains the ideal $\m\HH_0(A)$.}
with maximal ideal  $\m_{\HH_0(A)}$
\end{defn}

\begin{fact}\label{fact110211a}
Assume that $(R,\m)$ is local, and let $A$ be a local DG $R$-algebra.
The composition $A\to \HH_0(A)\to \HH_0(A)/\m_{\HH_0(A)}$
is a surjective morphism of DG $R$-algebras with kernel of the form
$\m_A=\cdots \xra{\partial_2^A} A_1 \xra{\partial_1^A} \mathfrak m_0\to 0$
for some maximal ideal $\mathfrak m_0\subsetneq A_0$.
The quotient complex $A/\m_A$ is $A$-isomorphic to $\HH_0(A)/\m_{\HH_0(A)}$.
Since $\HH_0(A)$ is a local $R$-algebra, we have $\m A_0\subseteq\m_0$.
\end{fact}

\begin{defn}\label{defn110216b}
If $R$ is local and $A$ is a local DG $R$-algebra, then the subcomplex $\m_A$ is the \emph{augmentation ideal} of $A$.
\end{defn}

The following example is key for this investigation.

\begin{ex}
Given a sequence $\ul t=t_1,\cdots,t_n\in R$,
the Koszul complex $K=K^R(\ul t)$ is a  DG $R$-algebra with product given by the wedge product.
If $(R,\m)$ is local and $\ul t\in\m$, then $K$ is a  local DG $R$-algebra
with augmentation ideal
$\m_K=(0\to R\to\cdots\to R^n\to\m\to 0)$.
\end{ex}

\begin{defn}
  \label{DGK2}
Let $A$ be a DG $R$-algebra. A 
\emph{DG $A$-module} is an $R$-complex $M$ with a
chain map
$\mu^M\colon \Otimes AM\to M$ such that the rule $am:=\mu^M(a\otimes m)$
is associative and unital.
The map $\mu^M$ is the \emph{scalar multiplication} on $M$.
A \emph{morphism} of DG $A$-modules is a chain map
$f\colon M\to N$ between DG $A$-modules that respects scalar multiplication:
$f(am)=af(m)$.
Isomorphisms in the category of DG $A$-modules are identified by the
symbol $\cong$. Quasiisomorphisms in the category of DG $A$-modules are identified by the symbol $\simeq$;
these are the morphisms that induce bijections on all homology modules.
Two DG $A$-modules $M$ and $N$ are \emph{quasiisomorphic}, writen $M\simeq N$
if there is a finite sequence of quasiisomorphisms  $M\xra\simeq\cdots\xla{\simeq}N$.
\end{defn}

\begin{ex}\label{ex110218a}
Consider the ring $R$ as a DG $R$-algebra.
A DG $R$-module is just an $R$-complex, and a morphism of DG $R$-modules is simply a chain map.
\end{ex}

\begin{fact}\label{fact110218a}
Let $A$ be a DG $R$-algebra, and let $M$ be a DG $A$-module.
The fact that the scalar multiplication on $M$ is a chain map
says that $\partial^M$ satisfies the \emph{Leibniz rule}:
$\partial^A_{|a|+|m|}(am)=\partial^A_{|a|}(a)m+(-1)^{|a|}a\partial^M_{|m|}(m)$.
The $R$-module $M_i$ is an $A_0$-module, and $\HH_i(M)$ is an $\HH_0(A)$-module
for each $i$.
\end{fact}

%\begin{fact}\label{fact110516a}
%Let $A$ be a DG $R$-algebra, and let $f\colon M\to N$ be a morphism of DG $A$-modules.
%Then the complex $\coker(f)$ has a well-defined DG $A$-module structure given by
%$a\ol n=\ol{an}$. It follows that the quotient $A/\m_A$ is a
%\end{fact}

\begin{defn}
  \label{defn110218a}
Let $A$ be a DG $R$-algebra, and let $i$ be an integer.
The $i$th \emph{suspension} of a DG $A$-module $M$
is the DG $A$-module $\shift^iM$ defined by $(\shift^iM)_n :=
M_{n-i}$ and $\partial^{\shift^iM}_n := (-1)^i\partial^M_{n-i}$. 
The scalar multiplication
on $\shift^iM$ is defined by the formula
$\mu^{\shift^iM}(a\otimes m):=(-1)^{i|a|}\mu^M(a\otimes m)$.
\end{defn}

\begin{defn}
\label{defn110223b}
Let $A$ be a DG $R$-algebra. A DG $A$-module $M$ is
\emph{homologically degreewise finite} if
$\HH_i(M)$ is finitely generated over $\HH_0(A)$ for all $i$;
it is
\emph{homologically finite} if
it is homologically degreewise finite
and $\HH_i(M)=0$ for $|i|\gg 0$.
\end{defn}

\begin{defn}
\label{defn110218c}
Let $A$ be a DG $R$-algebra, and let $M,N$ be  DG $A$-modules.
The \emph{tensor product} $\Otimes[A]MN$ is the quotient
$(\Otimes MN)/U$ where $U$ is the subcomplex generated by all elements of the form
$\Otimes[]{(am)}{n}-(-1)^{|a| |m|}\Otimes[]{m}{(an)}$.
Given an element $\Otimes[]mn\in\Otimes MN$, we denote the image in $\Otimes[A]MN$ as $\Otimes[]mn$.
\end{defn}

\begin{fact}\label{fact110218b}
Let $A$ be a DG $R$-algebra, and let $M,N$ be  DG $A$-modules.
The tensor product $\Otimes[A]MN$ is a DG $A$-module
via the action
$$a(\Otimes[] mn):=\Otimes[]{(am)}{n}=(-1)^{|a| |m|}\Otimes[]{m}{(an)}.$$
\end{fact}

\begin{fact}\label{fact110218c}
Let $A\to B$ be a morphism of DG $R$-algebras.
Given a DG $A$-module $M$, the
``base changed'' complex $\Otimes[A]BM$ has the structure of a DG $B$-module
by the action $b(\Otimes[] {b'}m):=\Otimes[]{(bb')}{m}$.
This structure is compatible with the DG $A$-module structure via restriction of scalars.
\end{fact}

\begin{defn}
  \label{defn110218b}
  Let $A$ be a DG $R$-algebra, and let $M$ be a DG $A$-module.
  The \emph{underlying $\und{A}$-module} associated to $M$ is the
$\und{A}$-module
$\und{M}=\oplus_{i=-\infty}^\infty M_i$.
A subset $E$ of $M$ is called a \emph{semi-basis} if it
is a basis of the underlying $A^\natural$-module $M^ \natural $.
If $M$ is bounded below, then $M$ is called \emph{semi-free}
if it has a semi-basis.\footnote{As is noted in~\cite{avramov:dgha}, when $M$ is not bounded below,
the definition of ``semi-free'' is significantly more technical.}
A \emph{semi-free resolution} of a DG $A$-module $N$ is a quasiisomorphism
$F\xra{\simeq}N$ of DG $A$-modules such that $F$ is semi-free.

Assume that $R$ and $A$ are local.
A \emph{minimal semi-free resolution} of $M$ is
a semi-free resolution $F\xra\simeq M$ such that $F$ is \emph{minimal}, i.e., 
each (equivalently, some) semi-basis of $F$ is finite in each degree and
the differential on
$\Otimes[A]{(A/\m_A)}{F}$ is 0.
\end{defn}

\begin{fact}\label{fact110218d}
Let $A$ be a DG $R$-algebra. Let $M$ be a homologically bounded below 
DG $A$-module.
Then $M$ has a semi-free resolution over $A$ by~\cite[Theorem 2.7.4.2]{avramov:dgha}.

Assume that $A$ is noetherian, and let $j$ be an integer.
If $\HH_i(M)$ is finitely generated over $\HH_0(A)$ for all $i$,
and $\HH_i(M)=0$ for $i<j$, then $M$ has a semi-free resolution $F\xra\simeq M$ such that
$\und{F}\cong\oplus_{i=j}^\infty \shift^i(\und{A})^{\beta_i}$
with $\beta_i\in\bbz$ for all $i$
and $F_i=0$ for all $i<j$;
see~\cite[Proposition~1]{apassov:hddgr}.
In particular, homologically finite DG $A$-modules admit such ``degreewise finite, bounded below'' semi-free resolutions.
Note that the condition $\und{F}\cong\oplus_{i=j}^\infty \shift^i(\und{A})^{\beta_i}$ says that the degree-$i$
piece of the semi-basis $E_i=E\cap F_i$ is finite for each $i$, and $E_i=\emptyset$ for $i<j$.

Assume that $R$ and $A$ are local. If $\HH_i(M)$ is finitely generated over $\HH_0(A)$ for all $i$,
and $\HH_i(M)=0$ for $i<j$, then $M$ has a minimal semi-free resolution $F\xra\simeq M$ such that $F_i=0$ for all $i<j$;
see~\cite[Proposition 2]{apassov:hddgr}.
In particular, homologically finite DG $A$-modules admit minimal semi-free resolutions.
\end{fact}

\begin{defn}
\label{defn110223a}
Let $A$ be a DG $R$-algebra, and let $M,N$ be  DG $A$-modules.
Given an integer $i$, a \emph{DG $A$-module homomorphism of degree $i$} is a homomorphism
$f\colon M\to N$ of the underlying $R$-complexes  such that
$f(am)=(-1)^{i|a|}af(m)$
for all $a\in A$ and $m\in M$.
We write $|f|=i$.
The (graded) submodule of $\Hom MN$
consisting of all DG $A$-module homomorphisms $M\to N$
is denoted $\Hom[A]MN$.
A homomorphism $f\in\Hom[A]MN_i$ is \emph{null-homotopic} if
it is in $\im(\partial^{\Hom[A]MN}_{i+1})$.
Two homomorphisms $M\to N$ are \emph{homotopic} if their difference is null-homotopic.
\end{defn}

\begin{fact}\label{fact110223a}
Let $A$ be a DG $R$-algebra, and let $M,N$ be  DG $A$-modules.
The complex $\Hom[A]MN$ is a DG $A$-module
via the action
$$(af)(m):=a (f(m))=(-1)^{|a||f|}f(am).
$$
\end{fact}

\begin{defn}
\label{defn110223a'}
Let $A\to B$ be a morphism of 
DG $R$-algebras, and let $M,M'$ be  DG $A$-modules.
Given  $f\in\Hom[A]M{M'}_i$, define
$\Otimes[A]Bf\colon\Otimes[A]BM\to\Otimes[B]B{M'}$ by the formula
$(\Otimes[A]Bf)(\Otimes[]bm)=(-1)^{i|b|}\Otimes[]b{f(m)}$.
\end{defn}

\begin{fact}\label{fact110223a'}
Let $A\to B$ be a morphism of 
DG $R$-algebras, and let $M,M'$ be  DG $A$-modules.
Given a homomorphism $f\in \Hom[A]M{M'}_i$,
the function $\Otimes[A]Bf$ is $B$-linear, that is, an element of $\Hom[B]{\Otimes[A]BM}{\Otimes[A]B{M'}}_i$.
Furthermore, if $f$ is a cycle in $\Hom[A]M{M'}_i$, then
$\Otimes[A]Bf$ is a cycle in $\Hom[B]{\Otimes[A]BM}{\Otimes[A]B{M'}}_i$.
\end{fact}

\begin{defn}
\label{defn110223a''}
Let $A$ be a DG $R$-algebra, and let $M,N$ be  DG $A$-modules.
Given a semi-free resolution $F\xra\simeq M$, we set 
$\Ext[A]iMN=\HH_{-i}(\Hom[A]FN)$ for each integer $i$.
\end{defn}

\begin{fact}\label{fact110223a''}
Let $A$ be a DG $R$-algebra, and let $M,N$ be  DG $A$-modules.
For each $i$, the module $\Ext[A]iMN$ is independent of the choice of semi-free resolution of $M$,
and we have $\Ext[A]iMN\simeq \Ext[A]i{M'}{N'}$ whenever $M\simeq M'$ and $N\simeq N'$;
see~\cite[Propositions 1.3.1--1.3.3]{avramov:ifr}.
Given a semi-free resolution $F\simeq M$ and an integer $i$, the elements of $\Ext[A]iMN$
are by definition the homotopy equivalence
classes of morphisms of DG $A$-modules $F\to\shift^{-i}N$.
\end{fact}

\section{Structure of Semi-free DG Modules and DG Homomorphisms}
\label{DG structures}

The proof of our Main Theorem involves the manipulation of 
the differentials on certain DG modules to construct
isomorphisms that are amenable to lifting. For this, we need
a concrete understanding of these differentials and the homomorphisms
between these DG modules. This concrete understanding is the goal of this section.
We begin by establishing some notation to be used
for much of the paper.

\begin{notn}\label{A,B}
%Suppose that
%$$
%A=0 \to A_d\xra {\partial^{A}_d}A_{d-1}
%\xra {\partial^{A}_{d-1}}\cdots\xra {\partial^{A}_1}A_0\to 0
%$$
Let $A$ be  a DG $R$-algebra such that each $A_i$ is free over $R$ of finite rank. Given an element $t\in R$, let
$K=K^R(t)$  denote the Koszul complex $0\to K_1\xra{t} K_0\to 0$
with $K_1\cong R\cong K_0$
and basis elements $1\in K_0$ and $e\in K_1$.
We fix a basis $\{\gamma_{i,1},\ldots,\gamma_{i,r_i}\}$ for
$A_i$. Let $B$  denote the DG $R$-algebra $K^R(t)\otimes_RA$,
which has the following form
$$
B\cong
\cdots \xra {\partial^{B}_{i+1}}
A_{i-1}
\oplus
A_i
\xra {\partial^{B}_{i}}
A_{i-2}
\oplus
A_{i-1}
\xra {\partial^{B}_{i-1}} \cdots \xra {\partial^{B}_{2}}
A_0
\oplus
A_1
\xra {\partial^{B}_{1}}
0
\oplus
A_0
\to 0.
$$
This uses the isomoprhism
$B_i=(K_1\otimes_RA_{i-1})\oplus(K_0\otimes_RA_{i})\cong A_{i-1}\oplus A_{i}$.
We identify $B_i$ with $A_{i-1}\oplus A_{i}$ for the remainder of this paper.
Under this identification, 
the sum $e\otimes a_{i-1}+1\otimes a_i\in B_i$
corresponds to the column vector $\left[\begin{smallmatrix}a_{i-1}\\ a_i\end{smallmatrix}\right]\in A_{i-1}\oplus A_{i}$.
The use of column vectors allows us to identify the differential of $B$
as the matrix
$$
\partial^B_i=
\left[
\begin{matrix}
-\partial^A_{i-1} & 0 \\
t &\partial^A_i
\end{matrix}
\right].
$$
\end{notn}

\begin{disc} \label{disc110311a}
In Notation~\ref{A,B},
the algebra structure on $B$ translates to the formula
$$\begin{bmatrix}a_{i-1}\\ a_i\end{bmatrix}\begin{bmatrix}c_{j-1}\\ c_j\end{bmatrix}
=\begin{bmatrix}a_{i-1}c_j+(-1)^ia_ic_{j-1}\\ a_ic_j\end{bmatrix}
$$
where $\left[\begin{smallmatrix}a_{i-1}\\ a_i\end{smallmatrix}\right]\in B_i$ and $\left[\begin{smallmatrix}c_{j-1}\\ c_j\end{smallmatrix}\right]\in B_j$.
This uses the fact that $e^2=0$ in $K$.

Note that a basis of $B_i$ is
$$\left\{\begin{bmatrix}\gamma_{i-1,1}\\0\end{bmatrix},\ldots,\begin{bmatrix}\gamma_{i-1,r_{i-1}}\\0\end{bmatrix},
\begin{bmatrix}0\\\gamma_{i,1}\end{bmatrix},\ldots,\begin{bmatrix}0\\\gamma_{i,r_i}\end{bmatrix}\right\}.$$
Also, note that the assumptions on $A$ imply that $A$ and $B$ are noetherian.
From the explicit description of $\partial^B$, it follows that $\HH_0(B)\cong\HH_0(A)/t\HH_0(A)$.

Assume that $R$ and $A$ are local.
Then $B$ is also local.
Moreover, given the augmentation ideal
$\m_A=\cdots \xra{\partial_2^A} A_1 \xra{\partial_1^A} \mathfrak m_0\to 0$
it is straightforward to show that the augmentation ideal of $B$ is
$$
\m_B=
\cdots \xra {\partial^{B}_{i+1}}
A_{i-1}
\oplus
A_i
\xra {\partial^{B}_{i}}
A_{i-2}
\oplus
A_{i-1}
\xra {\partial^{B}_{i-1}} \cdots \xra {\partial^{B}_{2}}
A_0
\oplus
A_1
\xra {\partial^{B}_{1}}
0
\oplus
\m_0
\to 0
$$
and we have $B/\m_B\cong A/\m_A$.
\end{disc}

\begin{notn}\label{notn110311a}
We work in the setting of Notation~\ref{A,B}.
Let
$\{\beta_i\}_{i=-\infty}^{\infty}$ be a set of cardinal numbers such that $\beta_i=0$ for $i\ll 0$.
For each integer $i$, set
$$M_i=\bigoplus_{j=0}^{\infty}A_j^{(\beta_{i-j})}$$
where $A_j^{(\beta_{i-j})}$ is a direct sum of
copies of $A_j$ indexed by $\beta_{i-j}$.
Identify each $\beta_i$ with a basis of $A_0^{(\beta_i)}$ over $A_0$, and set $\beta=\bigcup_i\beta_i$
considered as a subset of the disjoint union  $\bigsqcup_iM_i$.
Define scalar multiplication on $M$ over $A$ using the scalar multiplication on each $A^{(\beta_i)}$.

Consider $R$-module homomorphisms
\begin{align*}
\xi_i
&\colon M_{i}\to M_{i-1},
&\tau_i
&\colon M_{i}\to M_{i},
&
\delta_i
&\colon M_{i}\to M_{i-2},
&&\text{and}
&\alpha_i
&\colon M_{i}\to M_{i-1}.
\end{align*}
For each $i$, set
%\begin{align*}
%N_i&=
%\begin{matrix}
%K_1\otimes_R M_{i-1}\\
%\oplus\\
%K_0\otimes_R M_i
%\end{matrix}
%&&\text{and}
%&\partial_i^N
%&=
%\begin{bmatrix}
%K_1\otimes \xi_{i-1} & K_0\otimes \delta_i \\
%K_1\otimes \tau_{i-1} &K_0\otimes \alpha_i
%\end{bmatrix}.
%\end{align*}
\begin{align*}
N_i&=
M_{i-1}
\oplus
M_i
&&\text{and}
&\partial_i^N
&=
\begin{bmatrix}
\xi_{i-1} & \delta_i \\
\tau_{i-1} & \alpha_i
\end{bmatrix}\colon N_i\to N_{i-1}.
\end{align*}
We consider the sequences
$$
M=\cdots \xra{\alpha_{i+1}} M_i\xra{\alpha_{i}} M_{i-1}
\xra{\alpha_{i-1}} \cdots
$$
and
$$
N=
\cdots \xra{\partial^N_{i+1}} N_i\xra{\partial^N_{i}} N_{i-1}
\xra{\partial^N_{i-1}} \cdots.
%=
%\cdots\to
%\begin{matrix}
%K_1\otimes_R M_{i-1}\\
%\oplus\\
%K_0\otimes_R M_i
%\end{matrix}
%\to
%\begin{matrix}
%K_1\otimes_R M_{i-2}\\
%\oplus\\
%K_0\otimes_R M_{i-1}
%\end{matrix}
%\to \cdots
$$
%$$
%N= \cdots\to
%\begin{matrix}
%K_1\otimes_R M_{i-1}\\
%\oplus\\
%K_0\otimes_R M_i
%\end{matrix}
%\to
%\begin{matrix}
%K_1\otimes_R M_{i-2}\\
%\oplus\\
%K_0\otimes_R M_{i-1}
%\end{matrix}
%\to \cdots
%$$
Given elements
$\left[\begin{smallmatrix}a_{i-1}\\ a_i\end{smallmatrix}\right]\in B_i$ and $\left[\smash{\begin{smallmatrix}m_{j-1}\\ m_j\end{smallmatrix}}\right]\in N_j$,
we define
$$\begin{bmatrix}a_{i-1}\\ a_i\end{bmatrix}\begin{bmatrix}m_{j-1}\\ m_j\end{bmatrix}
=\begin{bmatrix}a_{i-1}m_j+(-1)^ia_im_{j-1}\\ a_im_j\end{bmatrix}.
$$
For each $\beta_{i,j}\in\beta_i$ we set $e_{i,j}=\left[\begin{smallmatrix}0\\ \beta_{i,j}\end{smallmatrix}\right]\in N_i$.
For each $i$, set $E_i=\{e_{i,j}\}_j$. Let $E=\bigcup_iE_i$
considered as a subset of the disjoint union  $\bigsqcup_iN_i$.
\end{notn}

\begin{disc}
In Notation~\ref{notn110311a}, the sequences $M$ and $N$ may not be  complexes. Note that the scalar multiplications defined on $M$ and $N$
make $\oplus_iM_i$ and $\oplus_iN_i$ into graded free modules over
$A^{\natural}$ and $B^{\natural}$, respectively.
\end{disc}

The next result is a straightforward consequence of the definitions in~\ref{notn110311a}.

\begin{lem}\label{lem110408b}
We work in the setting of Notations~\ref{A,B} and~\ref{notn110311a}.
The following conditions are equivalent.
\begin{enumerate}[\rm\quad(i)]
\item \label{lem110408b1}
The sequence $M$ is a semi-free DG $A$-module;
\item \label{lem110408b2}
The sequence $M$ is a  DG $A$-module; and
\item \label{lem110408b3}
For all integers $i$ and $j$ we have 
\begin{equation}\label{lem110408b4}
\alpha_{i-1}\alpha_i=0
\qquad\qquad
\alpha_{i+j}(\gamma_{i,s} m_j)=\partial_i^A(\gamma_{i,s})m_j+(-1)^i\gamma_{i,s}\alpha_j(m_j)
\end{equation}
for $s=1,\ldots,r_i$ and for all $m_j\in M_j$.
\end{enumerate}
\end{lem}

Next, we give a similar result for the sequence $N$.

\begin{lem}\label{lem110304a}
We work in the setting of Notations~\ref{A,B} and~\ref{notn110311a}.
The following conditions are equivalent.
\begin{enumerate}[\rm\quad(i)]
\item \label{lem110304a7}
The sequence $N$ is a semi-free DG $B$-module;
\item \label{lem110304a8}
The sequence $N$ is a  DG $B$-module; and
\item \label{lem110304a1}
For all integers $i$ and $j$  we have
\begin{align}
    \label{eq:s4}
\xi_i&=-\alpha_i
&\tau_i&=t
\\
    \label{eq:s4x}
\alpha_{i-1}\alpha_i&=-t\delta_i
&     \delta_i\alpha_{i+1}&=\alpha_{i-1}\delta_{i+1}
\end{align}
\vspace{-6mm}
\begin{gather}
 \label{lem110304a4}
  \delta_{i+j}(\gamma_{i,s} m_j)=\gamma_{i,s} \delta_j(m_j)\\
 \label{lem110304a3}
 \alpha_{i+j}(\gamma_{i,s} m_j)=\partial_i^A(\gamma_{i,s})m_j+(-1)^i\gamma_{i,s}\alpha_j(m_j)
%    \label{eq:s4}
%\xi_i&=-\alpha_i,
%&\tau_i&=t, &&\text{and}&
%      \alpha_{i-1}\alpha_i+t\delta_i&=0
% =     -\alpha_{i-1}\delta_{i+1}+\delta_i\alpha_{i+1}\\
% \label{lem110304a3}
% \alpha_{i+j}(\gamma_{i,s} m_j)=\partial_i^A(\gamma_{i,s})m_j+(-1)^i\gamma_{i,s}\alpha_j(m_j)
\end{gather}
for $s=1,\ldots,r_i$ and for all $m_j\in M_j$.
%Then $N$ is a semi-free DG $B$-module.
%\item \label{lem110304a2}
%Conversely, if $F$ is a bounded below semi-free DG $B$-module with semi-basis $G$,
%then $F\cong N$ for some appropriate choices of $\xi_i$, $\tau_i$, $\alpha_i$, and $\delta_i$ satisfying~\eqref{eq:s4}--\eqref{lem110304a3}
%where $\beta_i=|G\cap F_i|$.
\end{enumerate}
In particular, if $N$ is a DG $B$-module, then
\begin{equation}\partial_i^N=
\left[
\begin{matrix}
-\alpha_{i-1} & \delta_i \\
t & \alpha_i
\end{matrix}
\right].
\label{eq110413ax}
\end{equation}
\end{lem}

\begin{proof}
\eqref{lem110304a8}$\implies$\eqref{lem110304a1}
%We first show that if $N$ is a DG $B$-module, then the equations~\eqref{eq:s4}--\eqref{lem110304a3} are satisfied for each index $i$.
Assume that $N$ is a  DG $B$-module. Then the scalar multiplication defined in Notation~\ref{notn110311a} must satisfy the
Leibniz rule.
%Also we should note that for every element $a\otimes b$
%of $B$ and $c\otimes d$ of $N$ we have the following rule:
%\begin{equation}
%    \label{eq:Leibniz}
%    (a\otimes b)(c\otimes d)=(-1)^{|b||c|}ac\otimes bd.
%  \end{equation}
The Leibniz rule for products of the form $\left[\begin{smallmatrix}0\\ \gamma_{i,s}\end{smallmatrix}\right]\left[\begin{smallmatrix}0\\ m_j\end{smallmatrix}\right]$,
%(1\otimes \gamma_{i,s})(1\otimes m_j)$,
where $1\leq s\leq r_i$ and $m_j\in M_j$, is equivalent to the following relations:
\begin{align}
    \label{eq:s1}
      \delta_{i+j}(\gamma_{i,s} m_j)&=\gamma_{i,s} \delta_j(m_j)\\
    \label{eq:s1a}
      \alpha_{i+j}(\gamma_{i,s} m_j)&=\partial_i^A(\gamma_{i,s})m_j+(-1)^i\gamma_{i,s}\alpha_j(m_j).
\intertext{The Leibniz rule for products of the form
$\left[\begin{smallmatrix}\gamma_{i,s}\\ 0\end{smallmatrix}\right]\left[\begin{smallmatrix}0\\ m_j\end{smallmatrix}\right]$
%$(e\otimes \gamma_{i,s})(1\otimes m_j)$,
is equivalent to the following:}
    \label{eq:s2}
      \tau_{i+j}(\gamma_{i,s} m_j)&=t\gamma_{i,s} m_j\\
    \label{eq:s2a}
      \xi_{i+j}(\gamma_{i,s} m_j)&=-(\partial_i^A(\gamma_{i,s})m_j+(-1)^i\gamma_{i,s}\alpha_j(m_j)).
\intertext{The Leibniz rule for products of the form
$\left[\begin{smallmatrix}0\\\gamma_{i,s}\end{smallmatrix}\right]\left[\begin{smallmatrix} m_j\\ 0\end{smallmatrix}\right]$
is equivalent to the following:}
%And finally by writing the Leibniz rule for $(1\otimes \gamma_{i,s})(e\otimes m_j)$ we get:
    \label{eq:s3}
      \tau_{i+j}(\gamma_{i,s} m_j)&=\gamma_{i,s} \tau_j(m_j)\\
    \label{eq:s3a}
      \xi_{i+j}(\gamma_{i,s} m_j)&=-\partial_i^A(\gamma_{i,s})m_j+(-1)^i\gamma_{i,s}\xi_j(m_j).
\end{align}
The Leibniz rule for
$\left[\begin{smallmatrix}\gamma_{i,s}\\ 0\end{smallmatrix}\right]\left[\begin{smallmatrix}m_j\\ 0\end{smallmatrix}\right]=0$
%$(e\otimes \gamma_{i,s})(e\otimes m_j)(=0)$
is equivalent to the following:
\begin{equation}
(-1)^it\gamma_{i,s}m_j+(-1)^{i+1}\gamma_{i,s}\tau_j(m_j)=0.\label{eq:s3a'}
\end{equation}
%which is not a new relation.

Equation~\eqref{lem110304a4} is the same as~\eqref{eq:s1},
and equation~\eqref{lem110304a3} is the same as~\eqref{eq:s1a}.
Comparing equations~\eqref{eq:s1a} and~\eqref{eq:s2a} with $\gamma_{0,1}=1$, we find  $\xi_i=-\alpha_i$.
Using equation~\eqref{eq:s2} also with $\gamma_{0,1}=1$, we see that $\tau_i=t$.
This explains~\eqref{eq:s4x} and~\eqref{eq110413ax}.
It also shows that~\eqref{eq:s3a'} is trivial.
%$$\partial_i^N=
%\left[
%\begin{matrix}
%-K_1\otimes \alpha_{i-1} & K_0\otimes \delta_i \\
%t\otimes M_{i-1}&K_0\otimes \alpha_i
%\end{matrix}
%\right].
%$$
%\begin{equation}\partial_i^N=
%\left[
%\begin{matrix}
%-\alpha_{i-1} & \delta_i \\
%t & \alpha_i
%\end{matrix}
%\right].
%\label{eq110413a}
%\end{equation}
Since  $N$ is an $R$-complex,
we have $\partial^N_i\partial^N_{i+1}=0$ which gives  the
equations in~\eqref{eq:s4x}.
This completes the proof of the implication.

The implication
\eqref{lem110304a1}$\implies$\eqref{lem110304a8} is handled similarly,
and the equivalence \eqref{lem110304a7}$\iff$\eqref{lem110304a8}
is straightforward.
\end{proof}

Our next two results characterize semi-free DG modules over $A$ and $B$.
The first one is straightforward.

\begin{lem}\label{lem110429a}
We work in the setting of Notations~\ref{A,B} and~\ref{notn110311a}.
If $F$ is a bounded below semi-free DG $A$-module with semi-basis $G$,
then $F\cong M$ for some appropriate choice of $\alpha_i$   satisfying~\eqref{lem110408b4}  for all $i$ and $j$
where $\beta_i=|G\cap F_i|$.
\end{lem}

\begin{lem}\label{lem110304c}
We work in the setting of Notations~\ref{A,B} and~\ref{notn110311a}.
If $F$ is a bounded below semi-free DG $B$-module with semi-basis $G$,
then $F\cong N$ for some appropriate choices of $\xi_i$, $\tau_i$, $\alpha_i$, and $\delta_i$ satisfying~\eqref{eq:s4}--\eqref{lem110304a3} for all $i$ and $j$
where $\beta_i=|G\cap F_i|$.
\end{lem}

\begin{proof}
Let $F$ be a bounded below semi-free DG $B$-module with semi-basis $G$.
For each $i$, set $\beta_i=|G\bigcap F_i|$.
Since $F$ is semi-free, it is straightforward to show that
$F_i\cong\bigoplus_{j=0}^{\infty}B_j^{(\beta_{i-j})}$.
Decomposing $B_j$ as $A_{j-1}\bigoplus A_j$,
we see that $F_j\cong N_j$ for each~$j$.
Since the $R$-module homomorphisms $N_j\to N_{j-1}$ are necessarily of the form
$\left[\begin{smallmatrix}
\xi_{i-1} & \delta_i \\
\tau_{i-1} & \alpha_i
\end{smallmatrix}\right]$,
it follows that there are
appropriate choices of $\xi_i$, $\tau_i$, $\alpha_i$, and $\delta_i$
such that $F\cong N$.
Finally, the fact that $F$ is a DG $B$-module implies that the maps
$\xi_i$, $\tau_i$, $\alpha_i$, and $\delta_i$ satisfy~\eqref{eq:s4}--\eqref{lem110304a3}, by Lemma~\ref{lem110304a}.
\end{proof}

The next result indicates how a semi-free DG $B$-module should look in order to 
be liftable to $A$. See Section~\ref{lifting} for more about this.

\begin{lem}\label{lem110408a}
We work in the setting of Notations~\ref{A,B} and~\ref{notn110311a}.
If $M$ is a semi-free DG $A$-module,
then $B\otimes_AM$ is  a semi-free DG $B$-module,
identified with a DG $B$-module $N$ with
$$
\partial_i^{N}=
\left[
\begin{matrix}
-\alpha_{i-1} & 0 \\
t &\alpha_i
\end{matrix}
\right].
$$
\end{lem}

\begin{proof}
Using the isomorphisms
$$
B\otimes_AM\cong (K^R(t)\otimes_RA)\otimes_AM\cong K^R(t)\otimes_RM
$$
the result follows directly from the definitions in~\ref{notn110311a}
with Lemmas~\ref{lem110408b} and~\ref{lem110304a}.
\end{proof}

Next, we describe DG module homomorphisms over $A$ and $B$.
Again, the proof of the first of these results is straightforward.

\begin{lem}\label{lem110408d}
We work in the setting of Notations~\ref{A,B} and~\ref{notn110311a}.
Assume that $M$ is a semi-free DG $A$-module, and let $M'$ be a second semi-free DG $A$-module.
Fix an integer $p$.
A sequence of $R$-module homomorphisms $\{u_i\colon M_i\to M'_{i+p}\}$
is a DG $A$-module homomorphism $M\to M'$ of degree $p$ if and only if
it is a degree-$p$ homomorphism
$M\to M'$ of the underlying $R$-complexes  and
$$u_{i+j}(\gamma_{i,s}m_j)=(-1)^{pi}\gamma_{i,s}u_j(m_j)$$
for $s=1,\ldots,r_i$ and for all $m_j\in M_j$ for each integer $j$.
\end{lem}

\begin{lem}\label{lem110413a}
We work in the setting of Notations~\ref{A,B} and~\ref{notn110311a}.
Assume that $N$ is a semi-free DG $B$-module. Let $N'$ be a second semi-free DG $B$-module
built from modules $M_i'$ and maps $\xi_i'$, $\tau_i'$, $\delta_i'$, and $\alpha_i'$
as in~\ref{notn110311a}.
Fix an integer $p$.
A sequence of $R$-module homomorphisms
$\left\{
S_i\colon N_i\to N'_{i+p}
\right\}$
is a DG $B$-module homomorphism $N\to N'$ of degree $p$ if and only if
it is a degree-$p$ homomorphism
$N\to N'$ of the underlying $R$-complexes
such that for all integers $i$ we have
$S_i=\left[
\begin{smallmatrix}
(-1)^pz_{i-1} & v_i \\
0 & z_i
\end{smallmatrix}
\right]$
for some $z_i\colon M_i\to M'_{i+p}$ and $v_i\colon M_i\to M'_{i+p-1}$
and
%$$u_{i+j}(\gamma_{i,s}m_j)=(-1)^{pi}\gamma_{i,s}u_j(m_j)$$
\begin{align}
    \label{eq:h2'}
      v_{i+j}(\gamma_{i,s}m_j)&=(-1)^{i(p+1)}\gamma_{i,s}v_j(m_j)\\
      z_{i+j}(\gamma_{i,s}m_j)&=(-1)^{ip}\gamma_{i,s}z_j(m_j)     \label{eq:h2''}
\end{align}
for $s=1,\ldots,r_i$ and for all $m_j\in M_j$ for each integer $j$.
\end{lem}

\begin{proof}
Fix a sequence of $R$-module homomorphisms
$S=\left\{
S_i\colon N_i\to N'_{i+p}
\right\}$.
By assumption, we have
$N_i=
M_{i-1}
\oplus
M_i$
and
$N'_i=
M'_{i-1}
\oplus
M'_i$,
so the maps $S_i$ have the form
$$
S_i=
\left[
\begin{matrix}
u_{i-1} & v_i \\
y_{i-1} & z_i
\end{matrix}
\right]\colon
M_{i-1}
\oplus
M_i
\to
M'_{i+p-1}
\oplus
M'_{i+p}
$$
where
$u_{i-1}\colon M_{i-1}\to M'_{i+p-1}$,
$v_{i}\colon M_{i}\to M'_{i+p-1}$,
$y_{i-1}\colon M_{i-1}\to M'_{i+p}$, and
$z_{i}\colon M_{i}\to M'_{i+p}$.
%Set $S=\{S_i\}$.

Assume that $S$ is a DG $B$-module homomorphism $N\to N'$ of degree $p$.
For each $b_q\in B_q$ and $n_d\in N_d$,
we have
\begin{equation}
S_{q+d}(b_qn_d)=(-1)^{pq}b_qS_d(n_d).
\label{lem110413a1}
\end{equation}
Therefore, given integers $i$ and $j$, for  $s=1,\ldots,r_i$ and for all $m_j\in M_j$, by writing equation~\eqref{lem110413a1} for the elements
$\left[
\begin{smallmatrix}
\gamma_{i,s} \\
0
\end{smallmatrix}
\right]\in B_{i+1}$
and
$\left[
\begin{smallmatrix}
0 \\
m_j
\end{smallmatrix}
\right]\in N_{j}$
we have
\begin{align}
    \label{eq:h1}
      y_{i+j}(\gamma_{i,s}m_j)&=0\\
      u_{i+j}(\gamma_{i,s}m_j)&=(-1)^{(i+1)p}\gamma_{i,s}z_j(m_j).  \label{eq:h1'}
\end{align}
Using the elements
$\left[
\begin{smallmatrix}
0 \\
\gamma_{i,s}
\end{smallmatrix}
\right]\in B_{i}$
and
$\left[
\begin{smallmatrix}
0 \\
m_j
\end{smallmatrix}
\right]\in N_{j}$
we have
\begin{align}
    \label{eq:h2}
      v_{i+j}(\gamma_{i,s}m_j)&=(-1)^{i(p+1)}\gamma_{i,s}v_j(m_j)\\
      z_{i+j}(\gamma_{i,s}m_j)&=(-1)^{ip}\gamma_{i,s}z_j(m_j).     \label{eq:h2'''}
\end{align}
Similar equations arise using the elements
$\left[
\begin{smallmatrix}
\gamma_{i,s} \\ 0
\end{smallmatrix}
\right]\in B_{i+1}$
and
$\left[
\begin{smallmatrix}
m_j\\ 0
\end{smallmatrix}
\right]\in N_{j+1}$
and the elements
$\left[
\begin{smallmatrix}
0 \\
\gamma_{i,s}
\end{smallmatrix}
\right]\in B_{i}$
and
$\left[
\begin{smallmatrix}
m_j \\ 0
\end{smallmatrix}
\right]\in N_{j+1}$.

By comparing  equations~\eqref{eq:h1'} and~\eqref{eq:h2'''}  we conclude that $z_i=(-1)^{p}u_i$.
Equation~\eqref{eq:h1} with $\gamma_{0,1}=1$ implies that $y_i=0$ for all $i$. Therefore, $S_i$ has the desired form $S_i=\left[
\begin{smallmatrix}
(-1)^pz_{i-1} & v_i \\
0 & z_i
\end{smallmatrix}
\right]$.
Also, equations~\eqref{eq:h2} and~\eqref{eq:h2'''} are exactly~\eqref{eq:h2'} and~\eqref{eq:h2''}.
This completes the proof of the forward implication.
The converse is established similarly.
\end{proof}

The last result of this section describes some homomorphisms
between semi-free DG  $B$-modules that are liftable to $A$.

\begin{lem}\label{lem110429b}
We work in the setting of Notations~\ref{A,B} and~\ref{notn110311a}.
Let $M$ and $M'$ be semi-free DG $A$-modules, and let $f\in \Hom[A]M{M'}_j$.
If $(B\otimes_AM)_i$ is identified with $M_{i-1} \oplus  M_i$ and similarly for
$(B\otimes_AM')_i$,
then the map $(\Otimes[A]Bf)_i\colon (B\otimes_AM)_i\to (B\otimes_AM')_{i+j}$ is identified with
the matrix $\left[
\begin{smallmatrix}
(-1)^jf_{i-1} & 0 \\
0 & f_i
\end{smallmatrix}
\right]$.
%$$
%(\Otimes[A]Bf)_i=
%.
%$$
\end{lem}

\begin{proof}
This follows directly from Definition~\ref{defn110223a}.
\end{proof}

%%% SECTION
%%% 3 %%%%%%%%%%%%%%%%%%%%%%%%%%%%%%%%%%%%%%%%%%%%%%%%%%%%%%%%%%%
\section{Liftings and Quasi-liftings of  DG Modules}
\label{lifting}

In this section we prove our Main Theorem, starting with the definitions
of our notions of liftings in the DG arena.

\begin{defn}\label{lifting def}
Let $T\to S$ be a morphism of DG $R$-algebras, and let $D$ be a DG $S$-module.
Then $D$ is \emph{quasi-liftable} to
$T$ if there is a semi-free  DG $T$-module $D'$ such that
$D\simeq S\otimes_TD'$; in this case $D'$ is called a \emph{quasi-lifting}
of $D$ to $T$.
We say that $D$ is \emph{liftable} to
$T$ if there is a  DG $T$-module $D'$ such that
$D\cong S\otimes_TD'$; in this case $D'$ is called a \emph{lifting}
of $D$ to $T$.
\end{defn}

\begin{disc}
In the definition of ``quasi-liftable'' we require that $D'$ is semi-free
in order to avoid the need for derived categories. If one prefers,
one can remove the semi-free assumption and require that
$D\simeq S\lotimes_TD'$ instead.
\end{disc}

Our next result is a technical lemma for use in the proof of our Main Theorem.

\begin{lem}\label{lem120125a}
We work in the setting of Notations~\ref{A,B} and~\ref{notn110311a}.
Let $n\geq 1$, and let  $N^{(n-1)}$ be a semi-free DG $B$-module
$$
N^{(n-1)}= \cdots\to
M_{i-1}
\oplus
M_i
\xra {\left[
\begin{smallmatrix}
-\alpha^{(n-1)}_{i-1} & t^{n-1}\delta^{(n-1)}_i \\
t&\alpha^{(n-1)}_i
\end{smallmatrix}
\right]}
M_{i-2}
\oplus
M_{i-1}
\to \cdots
$$
whose semi-basis over $B$ is finite in each degree. 
In the case $n\geq 2$, assume that 
for each index $i$ there are $R$-module homomorphisms 
$v^{(n-2)}_i\colon M_i\to M_{i-2}$ and $z^{(n-2)}_i\colon M_i\to M_{i-1}$
such that
\begin{gather}
\alpha^{(n-1)}_{i}=\alpha^{(n-2)}_{i}+t^{n-1}z^{(n-2)}_{i}
\label{eq120710g}
\\
\delta_i^{(n-1)}=v^{(n-2)}_i-t^{n-2}z^{(n-2)}_{i-1}z^{(n-2)}_{i}
\label{eq120710a}
\\
\label{eq120710c}
      v^{(n-2)}_{i+j}(\gamma_{i,s}m_j)=\gamma_{i,s}v^{(n-2)}_j(m_j)\\
      z^{(n-2)}_{i+j}(\gamma_{i,s}m_j)=(-1)^{i}\gamma_{i,s}z^{(n-2)}_j(m_j).     
\label{eq120710d}
\\
\label{eq120710e}
      \alpha^{(n-2)}_{i-2}z^{(n-2)}_{i-1}+z^{(n-2)}_{i-2}\alpha^{(n-2)}_{i-1}+tv^{(n-2)}_{i-1}=\delta^{(n-2)}_{i-1}\\
      \!\!\!\!-\alpha^{(n-2)}_{i-2}v^{(n-2)}_i+t^{n-2}\delta^{(n-2)}_{i-1}z^{(n-2)}_i-t^{n-2}z^{(n-2)}_{i-2}\delta^{(n-2)}_i+v^{(n-2)}_{i-1}\alpha^{(n-2)}_i=0.
\label{eq120710f}
\end{gather}
for $s=1,\cdots,r_i$,  for all $m_j\in M_j$, and for each integer $j$.

If 
$\Ext[B]2{N^{(n-1)}}{N^{(n-1)}}=0$, then 
there is a semi-free DG $B$-module $N^{(n)}$
and an isomorphism of DG $B$-modules
$$\xymatrix@C=3.6em@R=4em{
N^{(n-1)}=\cdots
\ar[r]&
M_{i-1} \oplus  M_{i}
\ar[d]_{\left[\begin{smallmatrix}1 & -t^{n-1}z^{(n-1)}_{i} \\ 0 &
1\end{smallmatrix}\right]}\ar[rr]^{\left[\begin{smallmatrix}\!\!-\alpha^{(n-1)}_{i-1} & t^{n-1}\delta^{(n-1)}_{i}\!\! \\ t &
\alpha^{(n-1)}_{i} \end{smallmatrix}\right]}&&
M_{i-2} \oplus M_{i-1}
\ar[r]\ar[d]^{\left[\begin{smallmatrix}1 & -t^{n-1}z^{(n-1)}_{i-1} \\ 0 &
1 \end{smallmatrix}\right]}&
\cdots \\
N^{(n)}=\cdots
\ar[r]&
M_{i-1} \oplus  M_{i}
\ar[rr]^{\left[\begin{smallmatrix}-\alpha^{(n)}_{i-1} & t^n\delta^{(n)}_{i} \\ t &
\alpha^{(n)}_{i} \end{smallmatrix}\right]}&&
M_{i-2} \oplus M_{i-1}
\ar[r]&
\cdots.
}$$
such that for each index $i$ there are $R$-module homomorphisms 
$v^{(n-1)}_i\colon M_i\to M_{i-2}$ and $z^{(n-1)}_i\colon M_i\to M_{i-1}$
such that
\begin{gather}
\alpha^{(n)}_{i}=\alpha^{(n-1)}_{i}+t^{n}z^{(n-1)}_{i}
\label{eq120710h}
\\
\delta_i^{(n)}=v^{(n-1)}_i-t^{n-1}z^{(n-1)}_{i-1}z^{(n-1)}_{i}
\label{eq120710i}
\\
\label{eq120710j}
      v^{(n-1)}_{i+j}(\gamma_{i,s}m_j)=\gamma_{i,s}v^{(n-1)}_j(m_j)\\
      z^{(n-1)}_{i+j}(\gamma_{i,s}m_j)=(-1)^{i}\gamma_{i,s}z^{(n-1)}_j(m_j).     
\label{eq120710k}
\\
\label{eq120710l}
      \alpha^{(n-1)}_{i-2}z^{(n-1)}_{i-1}+z^{(n-1)}_{i-2}\alpha^{(n-1)}_{i-1}+tv^{(n-1)}_{i-1}=\delta^{(n-1)}_{i-1}\\
      \!\!\!\!-\alpha^{(n-1)}_{i-2}v^{(n-1)}_i+t^{n-1}\delta^{(n-1)}_{i-1}z^{(n-1)}_i-t^{n-1}z^{(n-1)}_{i-2}\delta^{(n-1)}_i+v^{(n-1)}_{i-1}\alpha^{(n-1)}_i=0.
\label{eq120710m}
\end{gather}
for $s=1,\cdots,r_i$,  for all $m_j\in M_j$, and for each integer $j$.
\end{lem}

\begin{proof}
By Lemma~\ref{lem110304a}, we conclude  that for all integers $i$, $j$  we have
\begin{gather}
    \label{eq110427h}
\alpha_{i-1}^{(n-1)}\alpha_i^{(n-1)}=-t^{n}\delta_i^{(n-1)}
\qquad\qquad
     t^{n-1}\delta_i^{(n-1)}\alpha_{i+1}^{(n-1)}=\alpha_{i-1}^{(n-1)}t^{n-1}\delta_{i+1}^{(n-1)}\\
 \label{eq110427i}
  t^{n-1}\delta_{i+j}^{(n-1)}(\gamma_{i,s} m_j)=\gamma_{i,s} t^{n-1}\delta_j^{(n-1)}(m_j)\\
 \label{eq110427j}
 \alpha_{i+j}^{(n-1)}(\gamma_{i,s} m_j)=\partial_i^A(\gamma_{i,s})m_j+(-1)^i\gamma_{i,s}\alpha_j^{(n-1)}(m_j)
\end{gather}
for $s=1,\ldots,r_i$ and for all $m_j\in M_j$.

Note that the sequence
$\left\{\left[
\begin{smallmatrix}
\delta^{(n-1)}_{i-1} & 0 \\
0 &\delta^{(n-1)}_i
\end{smallmatrix}
\right]\colon M_{i-1}  \oplus  M_{i} 
\to M_{i-3}  \oplus M_{i-2} 
\right\}$
is a cycle of degree $-2$ in the complex $\Hom[B]{N^{(n-1)}}{N^{(n-1)}}$.
Indeed, in the case $n=1$, this follows from
equations~\eqref{eq110427h}--\eqref{eq110427i};
in the case $n\geq 2$, this follows from
equations~\eqref{eq120710g}--\eqref{eq120710f}.
The assumption $\Ext[B]2{N^{(n-1)}}{N^{(n-1)}}=0$
implies that this cycle is null-homotopic, that is,
there is a DG $B$-module homomorphism $S^{(n-1)}=\{S^{(n-1)}_i\}\colon N^{(n-1)}\to N^{(n-1)}$ of
degree $-1$ such that
\begin{equation}
    \label{eq:ss3}
    \left[
\begin{smallmatrix}
\!\delta^{(n-1)}_{i-1} & 0 \\
0 &\delta^{(n-1)}_i\!
\end{smallmatrix}
\right]=
\left[
\begin{smallmatrix}
\!-\alpha^{(n-1)}_{i-2} & t^{n-1}\delta^{(n-1)}_{i-1}\! \\
t &\alpha^{(n-1)}_{i-1}
\end{smallmatrix}
\right]S^{(n-1)}_i+S^{(n-1)}_{i-1}\left[
\begin{smallmatrix}
\!-\alpha^{(n-1)}_{i-1} & t^{n-1}\delta^{(n-1)}_{i}\! \\
t &\alpha^{(n-1)}_{i}
\end{smallmatrix}
\right].
\end{equation}
Lemma~\ref{lem110413a} implies that each  $S^{(n-1)}_i$
is of the form
$$
S^{(n-1)}_i=\left[
\begin{smallmatrix}
-z^{(n-1)}_{i-1} & v^{(n-1)}_i \\
0 &z^{(n-1)}_i
\end{smallmatrix}
\right]
$$
where $v^{(n-1)}_i\colon M_i\to M_{i-2}$ and $z^{(n-1)}_i\colon M_i\to M_{i-1}$, and for $s=1,\cdots,r_i$, and for all $m_j\in M_j$ for each integer $j$ we have
\begin{align}
\label{eq:ss3'}
      v^{(n-1)}_{i+j}(\gamma_{i,s}m_j)&=\gamma_{i,s}v^{(n-1)}_j(m_j)\\
      z^{(n-1)}_{i+j}(\gamma_{i,s}m_j)&=(-1)^{i}\gamma_{i,s}z^{(n-1)}_j(m_j).     \label{eq:ss3''}
\end{align}
Hence for every $i$ the equality \eqref{eq:ss3} implies that we have
\begin{gather}
    \label{eq:ss4}
      \alpha^{(n-1)}_{i-2}z^{(n-1)}_{i-1}+z^{(n-1)}_{i-2}\alpha^{(n-1)}_{i-1}+tv^{(n-1)}_{i-1}=\delta^{(n-1)}_{i-1}\\
      \!\!\!\!-\alpha^{(n-1)}_{i-2}v^{(n-1)}_i+t^{n-1}\delta^{(n-1)}_{i-1}z^{(n-1)}_i-t^{n-1}z^{(n-1)}_{i-2}\delta^{(n-1)}_i+v^{(n-1)}_{i-1}\alpha^{(n-1)}_i=0.
    \label{eq:ss4'}
\end{gather}
Now let
\begin{align}
\alpha^{(n)}_{i}&=\alpha^{(n-1)}_{i}+t^nz^{(n-1)}_{i}&\delta_i^{(n)}&=v^{(n-1)}_i-t^{n-1}z^{(n-1)}_{i-1}z^{(n-1)}_{i}
\label{eq110420a}
\end{align}
and
$$
N^{(n)}= \cdots\to
M_{i-1}
\oplus
M_i
\xra {\left[
\begin{smallmatrix}
-\alpha^{(n)}_{i-1}& t^n\delta_i^{(n)} \\
t&\alpha^{(n)}_{i}
\end{smallmatrix}
\right]}
M_{i-2}
\oplus
M_{i-1}
\to \cdots.
$$
Note that the conclusions~\eqref{eq120710h}--\eqref{eq120710m} follow directly 
from~\eqref{eq:ss3'}--\eqref{eq110420a}.

Since $N^{(n-1)}$ is a DG $B$-module, equations~\eqref{eq110427h}, \eqref{eq:ss4}, and~\eqref{eq110420a} give the following equation for all $i$:
\begin{equation}\label{eq110427l}
\alpha^{(n)}_{i-1}\alpha^{(n)}_i+t^{n+1}\delta^{(n)}_i=0.
\end{equation}
By equations~\eqref{eq110427h}, \eqref{eq:ss4}, \eqref{eq:ss4'}, and~\eqref{eq110420a}, for all $i$ we have
\begin{equation}\label{eq110427m}
-\alpha^{(n)}_{i-1}t^n\delta^{(n)}_{i+1}+t^n\delta^{(n)}_i\alpha^{(n)}_{i+1}=0.
\end{equation}
For $s=1,\cdots,r_i$, and for all $m_j\in M_j$, equations~\eqref{eq110427i}, \eqref{eq:ss3'}, \eqref{eq:ss3''}, and~\eqref{eq110420a} give the following equality:
\begin{equation}\label{eq110427n}
t^n\delta^{(n)}_{i+j}(\gamma_{i,s}m_j)=\gamma_{i,s}t^n\delta^{(n)}_{j}(m_j).
\end{equation}
Also, by equations~\eqref{eq110427j}, \eqref{eq:ss3''}, and~\eqref{eq110420a} we conclude that
\begin{equation}\label{eq110427o}
\alpha^{(n)}_{i+j}(\gamma_{i,s}m_j)=\partial^A_i(\gamma_{i,s})m_j+(-1)^i \gamma_{i,s}\alpha^{(n)}_j(m_j).
\end{equation}
Therefore, equations~\eqref{eq110427l}--\eqref{eq110427o} and Lemma \ref{lem110304a} imply that $N^{(n)}$ is a semi-free DG $B$-module.
Equations~\eqref{eq:ss3''}--\eqref{eq:ss4} and~\eqref{eq110420a} provide the next morphism of DG $B$-modules:
$$\xymatrix@C=3.7em@R=4em{
N^{(n-1)}=\cdots
\ar[r]&
M_{i-1}  \oplus M_{i} 
\ar[d]_{\left[\begin{smallmatrix}1 & -t^{n-1}z^{(n-1)}_{i} \\ 0 &
1\end{smallmatrix}\right]}\ar[rr]^{\left[\begin{smallmatrix}\!\!-\alpha^{(n-1)}_{i-1} 
& t^{n-1}\delta^{(n-1)}_{i}\!\! \\ t &
\alpha^{(n-1)}_{i} \end{smallmatrix}\right]}&&
M_{i-2}  \oplus  M_{i-1} 
\ar[r]\ar[d]^{\left[\begin{smallmatrix}1 & -t^{n-1}z^{(n-1)}_{i-1} \\ 0 &
1 \end{smallmatrix}\right]}&
\cdots \\
N^{(n)}=\cdots
\ar[r]&
M_{i-1}  \oplus M_{i} 
\ar[rr]^{\left[\begin{smallmatrix}\!\!-\alpha^{(n)}_{i-1} & t^{n}\delta^{(n)}_{i}\!\! \\ t &
\alpha^{(n)}_{i} \end{smallmatrix}\right]}&&
M_{i-2}  \oplus  M_{i-1} 
\ar[r]&
\cdots.
}$$
Similar reasoning shows that the sequence
$\left\{\left[
\begin{smallmatrix}
1 & t^{n-1}z^{(n-1)}_{i} \\
0&1
\end{smallmatrix}
\right]\right\}$
is a morphism of DG $B$-modules, and it is straightforward to show that these sequences are inverse isomorphisms.
\end{proof}

Part~\eqref{item120131a} of our Main Theorem is a consequence of the next result.

\begin{thm}\label{lifting theorem}
We work in the setting of Notations~\ref{A,B} and~\ref{notn110311a}.
Assume that $R$ is $tR$-adically complete and that $N$ is semi-free
such that its semi-basis over $B$ is finite in each degree. If $\Ext[B]2NN=0$, then $N$ is
liftable to $A$ with semi-free lifting.
\end{thm}

\begin{proof}
Set $N^{(0)}=N$.
Here, we use a natural variation  of Notation~\ref{notn110311a}; see equation~\eqref{eq110413ax}:
$$
N^{(0)}= \cdots\to
M_{i-1}
\oplus
M_i
\xra {\left[
\begin{smallmatrix}
-\alpha^{(0)}_{i-1} & \delta^{(0)}_i \\
t&\alpha^{(0)}_i
\end{smallmatrix}
\right]}
M_{i-2}
\oplus
M_{i-1}
\to \cdots.
$$
Because of our assumptions, each $M_i$ is a finitely generated free $R$-module.

Lemma~\ref{lem120125a} implies that
for each $n\geq 1$ there is a semi-free DG $B$-module $N^{(n)}$
and an isomorphism of DG $B$-modules
$$\xymatrix@C=3.7em@R=4em{
N^{(n-1)}=\cdots
\ar[r]&
M_{i-1} \oplus  M_{i}
\ar[d]_{\left[\begin{smallmatrix}1 & -t^{n-1}z^{(n-1)}_{i} \\ 0 &
1\end{smallmatrix}\right]}\ar[rr]^{\left[\begin{smallmatrix}\!\!-\alpha^{(n-1)}_{i-1} & t^{n-1}\delta^{(n-1)}_{i}\!\! \\ t &
\alpha^{(n-1)}_{i} \end{smallmatrix}\right]}&&
M_{i-2} \oplus M_{i-1}
\ar[r]\ar[d]^{\left[\begin{smallmatrix}1 & -t^{n-1}z^{(n-1)}_{i-1} \\ 0 &
1 \end{smallmatrix}\right]}&
\cdots \\
N^{(n)}=\cdots
\ar[r]&
M_{i-1} \oplus  M_{i}
\ar[rr]^{\left[\begin{smallmatrix}-\alpha^{(n)}_{i-1} & t^n\delta^{(n)}_{i} \\ t &
\alpha^{(n)}_{i} \end{smallmatrix}\right]}&&
M_{i-2} \oplus M_{i-1}
\ar[r]&
\cdots
}$$
such that for each index $i$ there are $R$-module homomorphisms 
$v^{(n-1)}_i\colon M_i\to M_{i-2}$ and $z^{(n-1)}_i\colon M_i\to M_{i-1}$
such that
\begin{gather}
\alpha^{(n)}_{i}=\alpha^{(n-1)}_{i}+t^{n}z^{(n-1)}_{i}
\label{eq120710n}
\\
\delta_i^{(n)}=v^{(n-1)}_i-t^{n-1}z^{(n-1)}_{i-1}z^{(n-1)}_{i}
\label{eq120710o}
\\
\label{eq120710p}
      v^{(n-1)}_{i+j}(\gamma_{i,s}m_j)=\gamma_{i,s}v^{(n-1)}_j(m_j)\\
      z^{(n-1)}_{i+j}(\gamma_{i,s}m_j)=(-1)^{i}\gamma_{i,s}z^{(n-1)}_j(m_j).     
\label{eq120710q}
\\
\label{eq120710r}
      \alpha^{(n-1)}_{i-2}z^{(n-1)}_{i-1}+z^{(n-1)}_{i-2}\alpha^{(n-1)}_{i-1}+tv^{(n-1)}_{i-1}=\delta^{(n-1)}_{i-1}\\
      \!\!\!\!-\alpha^{(n-1)}_{i-2}v^{(n-1)}_i+t^{n-1}\delta^{(n-1)}_{i-1}z^{(n-1)}_i-t^{n-1}z^{(n-1)}_{i-2}\delta^{(n-1)}_i+v^{(n-1)}_{i-1}\alpha^{(n-1)}_i=0.
\label{eq120710s}
\end{gather}
for $s=1,\cdots,r_i$,  for all $m_j\in M_j$, and for each integer $j$.
This follows by induction on $n$; note that this uses the
isomorphism $N^{(n-1)}\cong N$ in the induction step to conclude  that
$\Ext[B]2{N^{(n-1)}}{N^{(n-1)}}\cong \Ext[B]2NN=0$.

A straightforward calculation shows that
$$\prod_{j= 0}^{n-1}
\left[
\begin{smallmatrix}
1 & -t^{j}z^{(j)}_{i} \\
0&1
\end{smallmatrix}
\right]=
\left[
\begin{smallmatrix}
1 & -\sum_{j= 0}^{n-1}t^{j}z^{(j)}_{i} \\
0&1
\end{smallmatrix}
\right].
$$
Hence, the composite isomorphism $N^{(0)}\to N^{(n)}$ has the following form:
$$\xymatrix@C=3.7em@R=4em{
N^{(0)}=\cdots
\ar[r]&
M_{i-1} \oplus  M_{i}
\ar[d]_{\left[
\begin{smallmatrix}
1 & -\sum_{j= 0}^{n-1}t^{j}z^{(j)}_{i} \\
0&1
\end{smallmatrix}
\right]}\ar[rr]^{\left[\begin{smallmatrix}\!\!-\alpha^{(0)}_{i-1} & \delta^{(0)}_{i}\!\! \\ t &
\alpha^{(0)}_{i} \end{smallmatrix}\right]}&&
M_{i-2}  \oplus  M_{i-1} 
\ar[r]\ar[d]^{\left[
\begin{smallmatrix}
1 & -\sum_{j= 0}^{n-1}t^{j}z^{(j)}_{i} \\
0&1
\end{smallmatrix}
\right]}&
\cdots \\
N^{(n)}=\cdots
\ar[r]&
M_{i-1}  \oplus  M_{i}
\ar[rr]^{\left[\begin{smallmatrix}\!\!-\alpha^{(n)}_{i-1} & t^{n}\delta^{(n)}_{i}\!\! \\ t &
\alpha^{(n)}_{i} \end{smallmatrix}\right]}&&
M_{i-2}  \oplus  M_{i-1} 
\ar[r]&
\cdots.
}$$
Furthermore,  equation~\eqref{eq120710n} shows that
$$
\alpha^{(n)}_i=\alpha_{i}^{(0)}+\sum_{j= 0}^{n-1}t^{j+1}z^{(j)}_{i}.
$$
We now define $N^{(\infty)}$ as follows:
$$
N^{(\infty)}= \cdots\to
M_{i-1}
\oplus
M_i
\xra {\left[
\begin{smallmatrix}
-\alpha^{(\infty)}_{i-1}& 0 \\
t&\alpha^{(\infty)}_{i}
\end{smallmatrix}
\right]}
M_{i-2}
\oplus
M_{i-1}
\to \cdots
$$
where
$$
\alpha^{(\infty)}_i=\alpha_{i}^{(0)}+\sum_{j= 0}^{\infty}t^{j+1}z^{(j)}_{i}.
$$
Note that $\alpha^{(\infty)}_i$ is well-defined because $R$ is $tR$-adically complete and the modules $M_i$  and $M_{i-1}$ are finitely generated free $R$-modules .

We claim  that $N^{(\infty)}$ is a semi-free DG $B$-module.
For all indices $i$ and $n$, set
\begin{equation}\label{eq110427p}
\zeta_i^{(n)}=\sum_{j=0}^\infty t^jz_i^{(j+n)}
\end{equation}
and notice that
\begin{equation}\label{eq110427q}
\alpha_i^{(\infty)}=\alpha_{i}^{(n)}+t^{n+1}\zeta_i^{(n)}.
\end{equation}
Using~\eqref{eq110427l}, it follows that
\begin{equation}\label{eq110427r}
\alpha_i^{(\infty)}\alpha_{i+1}^{(\infty)}
=t^{n+1}(-\delta_{i+1}^{(n)}+\alpha_{i}^{(n)}\zeta_{i+1}^{(n)}+\zeta_{i}^{(n)}\alpha_{i+1}^{(n)}+t^{n+1}\zeta_{i}^{(n)}\zeta_{i+1}^{(n)}).
\end{equation}
It follows that
$\alpha_i^{(\infty)}\alpha_{i+1}^{(\infty)}\in\bigcap_{n=1}^\infty t^{n+1}\Hom{M_i}{M_{i-1}}=0$,
by Krull's Intersection Theorem,
so we have
\begin{equation}\label{eq110427s}
\alpha_i^{(\infty)}\alpha_{i+1}^{(\infty)}=0.
\end{equation}
%Notice that equation~\eqref{eq:ss3''} holds for all $n$.
%Thus for all $n$, for $s=1,\cdots,r_i$, and for all $m_j\in M_j$ for each integer $j$ we have
%\begin{equation}\label{eq:ss3'''}
%z^{(n)}_{i+j}(\gamma_{i,s}m_j)=(-1)^{i}\gamma_{i,s}z^{(n)}_j(m_j).
%\end{equation}
Furthermore, for $s=1,\cdots,r_i$ and for all $m_j\in M_j$ equations~\eqref{eq110427o}, \eqref{eq120710q}, \eqref{eq110427p}, and \eqref{eq110427q} imply that
\begin{equation}\label{eq110427t}
\alpha_{i+j}^{(\infty)}(\gamma_{i,s}m_j)=\partial^A_i(\gamma_{i,s})m_j+(-1)^i\gamma_{i,s}\alpha_{j}^{(\infty)}(m_j).
\end{equation}
Thus by equations~\eqref{eq110427s}, \eqref{eq110427t} and Lemma \ref{lem110304a} we conclude that $N^{(\infty)}$ is a DG $B$-module.

Now consider the chain map
$\varphi=\{\varphi_i\}\colon N^{(0)}\to N^{(\infty)}$ defined by
$$\varphi_i=\prod_{j=0}^{\infty}
\left[
\begin{smallmatrix}
1 & -t^{j}z^{(j)}_{i} \\
0&1
\end{smallmatrix}
\right]
=
\left[
\begin{smallmatrix}
1 & -\sum_{j=0}^{\infty}t^{j}z^{(j)}_{i} \\
0&1
\end{smallmatrix}
\right]
$$
This map is well-defined because $R$ is $tR$-adically complete and $N_i$ is finitely generated over $R$.
Using these assumptions with equation~\eqref{eq120710q} we conclude that
\begin{equation*} %\label{eq:ss3'''}
\textstyle
\left(-\sum_{l=0}^{\infty}t^{l}z^{(l)}_{i+j}\right)(\gamma_{i,s}m_j)=(-1)^{i}\gamma_{i,s}\left(-\sum_{l=0}^{\infty}t^{l}z^{(l)}_{j}\right)(m_j)
\end{equation*}
for $s=1,\cdots,r_i$, and for all $m_j\in M_j$ for each integer $j$.
Thus $\varphi$ is $B$-linear and satisfies the assumptions of Lemma~\ref{lem110413a}, so $\varphi$ is a morphism of DG $B$-modules.
Similar reasoning shows that the sequence
$\left\{\left[
\begin{smallmatrix}
1 & \sum_{j=0}^{\infty}t^{j}z^{(j)}_{i} \\
0&1
\end{smallmatrix}
\right]\right\}$
is a morphism of DG $B$-modules, and it is easy to show that these sequences are inverse isomorphisms.

On the other hand, by equations~\eqref{eq110427s}--\eqref{eq110427t} and Lemma \ref{lem110408b}, the sequence
$$
M^{(\infty)}=\cdots\to M_i\xra{\alpha_i^{(\infty)}}M_{i-1}\to \cdots
$$
is a semi-free DG $A$-module. Now Lemma \ref{lem110408a} implies that $M^{(\infty)}$ is a lifting of $N^{(\infty)}$ to $A$.
Since $N\cong N^{(\infty)}$ by definition, we conclude that $M^{(\infty)}$ is a lifting of $N$ to $A$, so $N$ is liftable to $A$.
\end{proof}

\begin{cor}\label{cor120125a}
We work in the setting of Notation~\ref{A,B}.
Assume that $R$ is $tR$-adically complete. Let $D$ be a 
DG $B$-module that is 
homologically both bounded below and 
 degreewise finite. If $\Ext[B]2DD=0$, then $D$ is
quasi-liftable to $A$.
\end{cor}

\begin{proof}
Fact~\ref{fact110218d} 
implies that $D$ has a semi-free resolution
$N\simeq D$ over $B$
such  that the semi-basis for $N$ is finite in each degree.
Since $\Ext[B]2NN\cong \Ext[B]2DD=0$, 
Theorem~\ref{lifting theorem}
implies that $N$ is liftable to $A$, 
with semi-free lifting $M$.
Thus, we have $\Otimes[A]BM\cong N\simeq D$,
so $D$ is quasi-liftable to $A$.
\end{proof}

The proof of our Main Theorem uses induction on $n$, the length of the sequence $\ul t$.
The next result is useful for the induction step in this proof.

\begin{prop}\label{vanishing}
We work in the setting of Notation~\ref{A,B}. % and~\ref{notn110311a}
Assume that $R$ is $tR$-adically complete, and let $D$ be a
DG $B$-module that is homologically bounded below and homologically degreewise finite such that $\Ext[B]{d}DD=0$ for some  integer $d$.
If $M$ is a quasi-lifting of $D$ to $A$, then
$\Ext[A]dMM=0$.
\end{prop}

\begin{proof}
Assume without loss of generality  that 
$M$ is degreewise finite and bounded below; see
Fact~\ref{fact110218d}.  Lemma~\ref{lem110429a} shows that $M$ has the shape dictated by Notation~\ref{notn110311a}.
Since $M$ is a quasi-lifting of $D$ to $A$, we see that $N=\Otimes[A]BM\cong K^R(t)\otimes_R M$ is a semi-free resolution of $D$ over $B$; see Lemma~\ref{lem110408a}.
%In what follows, we use the isomorphism $\Otimes[A]{B}{\shift^{-d}M}\cong \shift^{-d}(\Otimes[A]{B}{M})=\shift^{-d}N$
%from Lemma~\ref{lem110503a}; see Remark~\ref{disc110503a}.
To show that $\Ext[A]dMM=0$,
%we use Fact~\ref{fact110223a}.
let $f =\{f_i\colon M_i\to M_{i-d}\}$ be a cycle in $\Hom[A]MM_{-d}$;
we need to show that $f$ is null-homotopic.
The fact that $f$ is a cycle says that
for every $i$ we have
$f_i\alpha_{i+1}=(-1)^d\alpha_{i+1-d}f_{i+1}$.
For each $i$ set $v^{(-1)}_i=f_i$

Claim: For all $n\geq 0$ and for all $i\in\bbz$, there are maps $v^{(n)}_i\colon M_i\to M_{i-d}$ and $z^{(n)}_i\colon M_i\to M_{i+1-d}$ such that for $s=1,\cdots,r_i$, and for all $m_j\in M_j$ for each  $j$
\begin{gather}
\label{eq:p1'''}
      v^{(n)}_{i+j}(\gamma_{i,s}m_j)=(-1)^{-id}\gamma_{i,s}v^{(n)}_j(m_j)\\
      z^{(n)}_{i+j}(\gamma_{i,s}m_j)=(-1)^{i(1-d)}\gamma_{i,s}z^{(n)}_j(m_j)     \label{eq:p2}\\
%\end{align}
%\vspace{-5mm}
%\begin{align}
    \label{eq:p3'}
      (-1)^dz^{(n)}_{i-2}\alpha_{i-1}+tv^{(n)}_{i-1}+\alpha_{i-d}z^{(n)}_{i-1}=v^{(n-1)}_{i-1}\\
      (-1)^dv^{(n)}_{i-1}\alpha_{i}-\alpha_{i-d}v_i^{(n)}=0. \label{eq:p4'}
\end{gather}

To prove the claim, we proceed by induction on $n$.
We verify the base case and the inductive step simultaneously.
Let $n\geq 0$ and assume that for each $i$ there exists $v^{(n-1)}_i\colon M_i\to M_{i-d}$ such that for $s=1,\cdots,r_i$, and for all $m_j\in M_j$ for each integer $j$,
we have
\begin{gather}
\label{eq:p1''''}
      v^{(n-1)}_{i+j}(\gamma_{i,s}m_j)=(-1)^{id}\gamma_{i,s}v^{(n-1)}_j(m_j) \\
%      z^{(n-1)}_{i+j}(\gamma_{i,s}m_j)&=(-1)^{i}\gamma_{i,s}z^{(n-1)}_j(m_j)     \label{eq:p2''}
%\end{align}
%and  we have
%\begin{align}
%    \label{eq:p3''}
%      z^{(n-1)}_{i-2}\alpha_{i-1}+tv^{(n-1)}_{i-1}+(-1)^d\alpha_{i-d}z^{(n-1)}_{i-1}&=v^{(n-2)}_{i-1}\\
      v^{(n-1)}_{i-1}\alpha_{i}-(-1)^d\alpha_{i-d}v_i^{(n-1)}=0. \label{eq:p4''}
\end{gather}
%By
%equations~\eqref{eq:p4''}, \eqref{eq:p1''''} and
Thus, the sequence
$\left\{\left[\begin{smallmatrix}(-1)^dv^{(n-1)}_{i-1} & 0 \\ 0 &
v^{(n-1)}_{i}\end{smallmatrix}\right]
\colon M_{i-1}  \oplus  M_{i} 
\to M_{i-d-1}  \oplus  M_{i-d} 
\right\}$ is a cycle in $\Hom[B]NN_{-d}$,
by Lemma~\ref{lem110413a}.
%provide the following morphism of DG $B$-modules
%$$\xymatrix@C=4.84em@R=4em{
%N=
%\cdots\ar[r]
%&\protect{\begin{matrix}M_{i-1} \\ \oplus \\ M_{i} \end{matrix}}
%\ar[d]_{\left[\begin{smallmatrix}v^{(n-1)}_{i-1} & 0 \\ 0 &
%v^{(n-1)}_{i}\end{smallmatrix}\right]}\ar[rr]^{\left[\begin{smallmatrix}-\alpha_{i-1} & 0 \\ t &
%\alpha_{i} \end{smallmatrix}\right]}&&
%\protect{\begin{matrix}M_{i-2} \\ \oplus \\ M_{i-1} \end{matrix}}
%\ar[d]^{\left[\begin{smallmatrix}v^{(n-1)}_{i-2} & 0 \\ 0 &
%v^{(n-1)}_{i-1} \end{smallmatrix}\right]}
%\ar[r]&\cdots
%\\
%\shift^{-d}N=
%\cdots\ar[r]
%&\protect{\begin{matrix}M_{i-1-d} \\ \oplus \\ M_{i-d} \end{matrix}}
%\ar[rr]^{\left[\begin{smallmatrix}\!\!-(-1)^{-d}\alpha_{i-1-d} & 0 \\ t &
%(-1)^{-d}\alpha_{i-d}\! \end{smallmatrix}\right]}&&
%\protect{\begin{matrix}M_{i-2-d} \\ \oplus \\ M_{i-1-d} \end{matrix}}
%\ar[r]&\cdots.
%}$$
As $\Ext[B]{d}DD=0$, this morphism is null-homotopic.
Thus there exists a DG $B$-module homomorphism $S^{(n)}=\{S_i^{(n)}\}\in\Hom[B]NN_{1-d}$
such that for every $i$ we have
\begin{equation}\label{eq:p1''}
\left[\begin{smallmatrix}-\alpha_{i-d} & 0 \\ t &
\alpha_{i-d+1} \end{smallmatrix}\right]S_i^{(n)}-(-1)^{1-d}S_{i-1}^{(n)}\left[\begin{smallmatrix}-\alpha_{i-1} & 0 \\ t &
\alpha_{i} \end{smallmatrix}\right]=\left[\begin{smallmatrix}(-1)^dv^{(n-1)}_{i-1} & 0 \\ 0 &
v^{(n-1)}_{i} \end{smallmatrix}\right].
\end{equation}
Lemma~\ref{lem110413a} implies that each $S^{(n)}_i$
is of the form
$$
S_i^{(n)}=\left[\begin{smallmatrix}(-1)^{1-d}z^{(n)}_{i-1} & v_{i}^{(n)} \\ 0 &
z^{(n)}_{i} \end{smallmatrix}\right]:N_i\to N_{i-d+1}
$$
where $v^{(n)}_i\colon M_i\to M_{i-d}$ and $z^{(n)}_i\colon M_i\to M_{i+1-d}$, and for $s=1,\cdots,r_i$, and for all $m_j\in M_j$ for each integer $j$ we have
\begin{align}
\label{eq:p5}
      v^{(n)}_{i+j}(\gamma_{i,s}m_j)&=(-1)^{-id}\gamma_{i,s}v^{(n)}_j(m_j)\\
      z^{(n)}_{i+j}(\gamma_{i,s}m_j)&=(-1)^{i(1-d)}\gamma_{i,s}z^{(n)}_j(m_j).     \label{eq:p6}
\end{align}
Hence for each $i$, equation~\eqref{eq:p1''} implies that we have
\begin{align}
    \label{eq:p7}
      (-1)^dz^{(n)}_{i-2}\alpha_{i-1}+tv^{(n)}_{i-1}+\alpha_{i-d}z^{(n)}_{i-1}&=v^{(n-1)}_{i-1}\\
      (-1)^dv^{(n)}_{i-1}\alpha_{i}-\alpha_{i-d}v_i^{(n)}&=0. \label{eq:p8}
\end{align}
This completes the proof of the claim.

Equation~\eqref{eq:p3'} implies the following equality for each $i$:
$$
f_{i}=(-1)^d\Bigg[\sum_{j=0}^{n}t^{j}z^{(j)}_{i-1}\Bigg]\alpha_{i}+\alpha_{i+1-d}\Bigg[\sum_{j=0}^{n}t^{j}z^{(j)}_{i}\Bigg]+t^{n+1}v_{i}^{(n)}.
$$
%If we let $n\to \infty$, s
Since $R$ is $tR$-adically complete and each $M_i$ is finitely generated over $R$, each series $\eta_i=\sum_{j=0}^{\infty}t^{j}z^{(j)}_{i}$
converges in $\Hom{M_i}{M_{i+1-d}}$, and for every $i$ we have
\begin{equation}\label{eq110503a}
f_{i}=(-1)^d\eta_{i-1}\alpha_{i}+\alpha_{i+1-d}\eta_i.
\end{equation}
By equation~\eqref{eq:p2}, we conclude that
$\eta_{i+j}(\gamma_{i,s}m_j)=(-1)^{i(1-d)}\gamma_{i,s}\eta_j(m_j)$
for all $i,j,s$.
Thus, Lemma~\ref{lem110408d} implies that $\eta=\{\eta_i\}\in\Hom[A]MM$ is a DG $A$-module homomorphism of degree
$1-d$. Equation~\eqref{eq110503a} implies that $f=\{f_i\}$ is null-homotopic, as desired.
\end{proof}

The next result contains part~\eqref{item120131a} of our Main Theorem.

\begin{cor}\label{cor110513a}
Let $\underline{t}=t_1,\cdots,t_n$ be a sequence of elements of $R$, and
assume that $R$ is $\underline tR$-adically  complete.
Let $D$ be a
DG $K^R(\ul t)$-module that is homologically bounded below and homologically degreewise finite. If $\Ext[K^R(\ul t)]{2}DD=0$,
then $D$ is quasi-liftable to $R$.
\end{cor}

\begin{proof}
%The existence of a lifting of $D$ to $R$ follows
By induction on $n$, using
Corollary~\ref{cor120125a} and Proposition~\ref{vanishing}.
\end{proof}

The version of our Main Theorem  used in~\cite{nasseh:lrfsdc} 
requires a bit more terminology.

\begin{defn}
\label{dfn:DGsdm}
Let $A$ be a DG $R$-algebra, and let $M$ be a DG $A$-module.
For each $a\in A$ the \emph{homothety}
(i.e., multiplication map) $\mult Ma\colon M\to M$
given by $m\mapsto am$
is a homomorphism of degree $|a|$.
The \emph{homothety morphism} $X^A_M\colon A\to \HomA MM$ is given by
$X^A_M(a)=\mult Ma$, i.e., $X^A_M(a)(m)=am$.
%This induces a \emph{homothety morphism} $\chi^M\colon A\to \RhomA MM$.

If $A$ is noetherian and $M$ is semi-free, then  $M$ is a \emph{semidualizing}
DG $A$-module if $M$ is homologically finite
and the homothety morphism
$X^A_M\colon A\to\Hom[A]{M}{M}$ is a quasiisomorphism.
We say that an $R$-complex is \emph{semidualizing} provided that it is
semidualizing as a DG $R$-module.
Let $\s(A)$ be the set of shift-quasiisomorphism classes of semidualizing DG $A$-modules,
that is, the set of equivalence classes of semidualizing DG $A$-modules where $C\sim C'$ if there is an
integer $n$ such that $C'\simeq\shift^nC$.
\end{defn}

\begin{disc}
Let $A$ be a DG $R$-algebra, and let $M$, $M'$ be  
semi-free DG $A$-modules.
It is straightforward to show that
if $M\sim M'$, then $M$ is semidualizing if and only if $M'$ is semidualizing.
Note that our semi-free assumption in the definition of semidualizing
is only made in order to avoid the need for the derived category $\catd(A)$.
If one prefers to work in $\catd(A)$, then a homologically finite DG $A$-module
$N$ is semidualizing if (and only if) the induced map
$\chi_N^A\colon A\to\Rhom[A]NN$ is an isomorphism in $\catd(A)$.
\end{disc}

\begin{cor}\label{lifting to R}
Let $\underline{t}=t_1,\cdots,t_n$ be a sequence of elements of $R$, and
assume that $R$ is $\underline tR$-adically  complete.
If $D$ is a semidualizing DG $K^R(\underline{t})$-module, then there
exists a semidualizing $R$-complex $C$ which is a quasi-lifting of $D$ to $R$.
Moreover, the base-change operation $C\mapsto\Otimes{K^R(\underline t)}C$
induces a bijection $\s(R)\xra\cong\s(K^R(\underline t))$.
\end{cor}

\begin{proof}
Note that the fact that $R$ is $\underline tR$-adically  complete implies that
$\underline tR$ is contained in the Jacobson radical of $R$.
Using this, one checks readily that the conclusions 
of~\cite[Lemma A.3]{christensen:dvke} hold in our setting.
The existence of an $R$-complex $C$ that is a quasi-lifting of $D$ to $R$ follows
from Corollary~\ref{cor110513a}; and $C$ is  semidualizing over $R$
by~\cite[Lemma A.3(a)]{christensen:dvke}. This says that the
base-change map $\s(R)\to\s(K^R(\underline t))$ is surjective;
it is injective by~\cite[Lemma A.3(b)]{christensen:dvke}.
\end{proof}

Part~\eqref{item120131b} of our Main Theorem is a consequence of the next result.

\begin{thm}\label{uniqueness}
We work in the setting of Notation~\ref{A,B}.
Assume that $R$ is $tR$-adically complete, and assume that 
$R$, $A$, and $A_0$ are local, and let $D$ be a
DG $B$-module that is homologically bounded below
and homologically degreewise finite. If $D$ is quasi-liftable
to $A$ and $\Ext[B]{1}DD=0$, then any two 
homologically degreewise finite quasi-lifts of $D$ to $A$ are quasiisomorphic
over $A$.
\end{thm}

\begin{proof}
The assumption that $R$ is $tR$-adically complete and local implies that
$t$ is in the maximal ideal $\m\subset R$.

Let $C$ and $C'$ be two homologically degreewise finite semi-free DG $A$-modules such that $\Otimes[A]BC\simeq D\simeq \Otimes[A]BC'$. Let $M\xra\simeq C$ and $M'\xra\simeq C'$ be minimal
semi-free resolutions of $C$ and $C'$ over $A$. Lemma~\ref{lem110429a} shows that $M$ and $M'$ have the shape dictated by Notation~\ref{notn110311a}.
Since $C$ and $C'$ are quasi-liftings of $D$ to $A$, we see that 
$N:=\Otimes[A]BM\cong K^R(t)\otimes_R M$ and 
$N':=\Otimes[A]BM'\cong K^R(t)\otimes_R M'$ are
semi-free resolutions of $D$ over $B$; see Lemma~\ref{lem110408a}.
Furthermore, from Remark~\ref{disc110311a}, we have the isomorphism $B/\m_B\cong A/\m_A$, which implies that
$$\Otimes[B]{B/\m_B}{(\Otimes[A]{B}{M})}\cong \Otimes[A]{A/\m_A}{M}.$$
Since $M$ is minimal over $A$, the differential on this complex is 0, so $\Otimes[A]{B}{M}$ is minimal over $B$,
and similarly for $\Otimes[A]{B}{M'}$.

From~\cite[Theorem 2.12.5.2 and Example 2.12.5.4]{avramov:dgha} there exists  an isomorphism $\Upsilon\colon N\xra\cong N'$.
Lemma~\ref{lem110413a} implies that $\Upsilon$ has the following form
%To see $M\simeq M'$ as semi-free DG $A$-modules, consider the DG quasiisomorphism
\begin{equation}\label{eq110516a}
\begin{split}
\xymatrix@C=3em@R=4em{
N=\cdots
\ar[r]&
M_{i-1}  \oplus  M_{i} 
\ar[d]_{\left[\begin{smallmatrix}z_{i-1}& v_{i} \\ 0 &
z_i\end{smallmatrix}\right]}\ar[rr]^{\left[\begin{smallmatrix}-\alpha_{i-1} & 0 \\ t &
\alpha_{i} \end{smallmatrix}\right]}&&
M_{i-2}  \oplus  M_{i-1} 
\ar[r]\ar[d]^{\left[\begin{smallmatrix}z_{i-2} & v_{i-1} \\ 0 &
z_{i-1} \end{smallmatrix}\right]}&
\cdots \\
N'=\cdots
\ar[r]&
M'_{i-1}  \oplus  M'_{i} 
\ar[rr]^{\left[\begin{smallmatrix}-\alpha'_{i-1} & 0 \\ t &
\alpha'_{i} \end{smallmatrix}\right]}&&
M'_{i-2}  \oplus  M'_{i-1} 
\ar[r]&
\cdots
}\end{split}\end{equation}
and that we have
\begin{align}
v_{i+j}(\gamma_{i,s}m_j)&=(-1)^i\gamma_{i,s}v_j(m_j) \label{eq110516b} \\
      z_{i+j}(\gamma_{i,s}m_j)&=\gamma_{i,s}z_j(m_j)     \label{eq110506b}
\end{align}
for all $i$, for $s=1,\ldots,r_i$ and for all $m_j\in M_j$ for each integer $j$.
As the diagram~\eqref{eq110516a} commutes, we have
\begin{align}
    \label{eq:u1}
      z_{i-1}\alpha_{i}&=tv_{i}+\alpha'_{i}z_{i}
\end{align}
for all $i$. Since $\Upsilon$ is an isomorphism, it follows  that the $z_i$'s are isomorphisms.

The condition~\eqref{eq110516b} implies that
$v=\{v_i\}\in\Hom[A]{M}{M'}_{-1}$; cf.~Lemma~\ref{lem110408d}. As the diagram~\eqref{eq110516a} commutes, we have
$v_{i-1}\alpha_{i}=-\alpha'_{i-1}v_i$
for all $i$, so $v$ is a cycle in $\Hom[A]M{M'}_{-1}$.
This yields a cycle $\Otimes[A]{B}{v}\in \Hom[B]{N}{N'}_{-1}$
which
has the form
$\left\{\left[\begin{smallmatrix}-v_{i-1} & 0 \\ 0 &
v_{i}\end{smallmatrix}\right]
\colon M_{i-1}  \oplus  M_{i} \to
M_{i-d-1}  \oplus  M_{i-d}\right\}$;
see Fact~\ref{fact110223a} and Lemma~\ref{lem110429b}.
Since $\Otimes[A]{B}{v}$ is a cycle,
our Ext-vanishing assumption
implies that  there is a DG $B$-module homomorphism
$$
T^{(0)}=\left\{\left[\begin{smallmatrix}u^{(0)}_{i-1}& p_i^{(0)} \\ 0 &
u^{(0)}_{i}\end{smallmatrix}\right]\right\}\in \Hom[B]{N}{N'}_{0}
$$
such that for every $i$ we have
$$
\left[\begin{smallmatrix}-\alpha'_{i-1}& 0 \\ t &
\alpha'_i\end{smallmatrix}\right]T^{(0)}_{i}-T^{(0)}_{i-1}\left[\begin{smallmatrix}-\alpha_{i-1}& 0 \\ t &
\alpha_i\end{smallmatrix}\right]=\left[\begin{smallmatrix}-v_{i-1}& 0 \\ 0 &
v_i\end{smallmatrix}\right].
$$
Therefore for all $i\in\bbz$  we obtain the following equations:
\begin{gather}
%    \label{eq:u2}
      -u_{i-1}^{(0)}\alpha_{i}+tp_{i}^{(0)}+\alpha'_{i}u_{i}^{(0)}=v_i \notag \\
      u_{i+j}^{(0)}(\gamma_{i,s}m_j)= \gamma_{i,s}u_{j}^{(0)}(m_j)\notag\\
p_{i+j}^{(0)}(\gamma_{i,s}m_j)=(-1)^i\gamma_{i,s}p_i^{(0)}(m_j)\notag \\
      p_{i-1}^{(0)}\alpha_{i}=-\alpha'_{i-1}p_i^{(0)}.\notag
\end{gather}
%Note that the second equation here is somewhat subtle.
%We explain it using the notation from Definition~\ref{defn110218a}:
%$u_{i+j}^{(0)}(\gamma_{i,s}m_j)=(-1)^i \gamma_{i,s}(\shift u_{j}^{(0)}(m_j))=(-1)^{2i} \gamma_{i,s}u_{j}^{(0)}(m_j)= \gamma_{i,s}u_{j}^{(0)}(m_j)$.

The process repeats using $p^{(0)}=\{p^{(0)}_i\}$ in place of $p^{(-1)}=v$.
Inductively, for each $n\geq 0$ one can construct
a DG $B$-module homomorphism
$$
T^{(n)}=\left\{\left[\begin{smallmatrix}u^{(n)}_{i-1}& p_i^{(n)} \\ 0 &
u^{(n)}_{i}\end{smallmatrix}\right]\right\}\in \Hom[B]{N}{N'}_{0}
$$
such that for every $i$ we have
$$
\left[\begin{smallmatrix}-\alpha'_{i-1}& 0 \\ t &
\alpha'_i\end{smallmatrix}\right]T^{(n)}_{i}-T^{(n)}_{i-1}\left[\begin{smallmatrix}-\alpha_{i-1}& 0 \\ t &
\alpha_i\end{smallmatrix}\right]=\left[\begin{smallmatrix}-p^{(n-1)}_{i-1}& 0 \\ 0 &
p_i^{(n-1)}\end{smallmatrix}\right].
$$
Therefore for all $i\in\bbz$ and $n\geq 0$ we get the following equations:
%\begin{equation}
%    \label{eq:u2}
%      u_{i-1}^{(0)}\alpha_{i}+tp_{i}^{(0)}-\alpha'_{i}u_{i}^{(0)}=v_i
%%      p_{i-1}^{(0)}\alpha_{i}=-\alpha'_{i-1}p_i^{(0)}.
%\end{equation}
%By considering the second equation and repeating the above process after $n$ times we get the following
%equations
\begin{gather}
%    \label{eq:u3}
%    \begin{cases}
      -u_{i-1}^{(n)}\alpha_{i}+tp_{i}^{(n)}+\alpha'_{i}u_{i}^{(n)}=p_i^{(n-1)}
      \nonumber \\
      u_{i+j}^{(n)}(\gamma_{i,s}m_j)= \gamma_{i,s}u_{j}^{(n)}(m_j) \label{eq110506c}
%      p_{i-1}^{(n)}\alpha_{i}=-\alpha'_{i-1}p_i^{(n)}.
%    \end{cases}
\end{gather}
and hence
\begin{equation}
\label{eq110506a}
v_i=p^{(-1)}_i=\alpha'_i\Bigg[\sum_{j=0}^{n}t^ju_{i}^{(j)}\Bigg]+t^{n+1}p_i^{(n)}-\Bigg[\sum_{j=0}^{n}t^ju_{i-1}^{(j)}\Bigg]\alpha_{i}.
\end{equation}
Since $R$ is $tR$-adically complete, the next series converges for each $i$
$$
\xi_i=\sum_{j=0}^{\infty}t^{j}u_{i}^{(j)}
$$
and equation~\eqref{eq110506a} implies that
\begin{equation}
    \label{eq:u4}
      v_i=\alpha'_i\xi_{i}-\xi_{i-1}\alpha_{i}.
\end{equation}
Combining equations~\eqref{eq:u1} and~\eqref{eq:u4},  for each $i$ we have 
\begin{equation}
    \label{eq:u5}
      (z_{i-1}+t\xi_{i-1})\alpha_{i}=\alpha'_i(z_i+t\xi_i).
\end{equation}
 This shows that the sequence
 $z+t\xi\colon M\to M'$ is a degree-0 homomorphism of the underlying $R$-complexes.
% DG $A$-module homomorphism $M\to M'$ of degree 0.
 Combining equations~\eqref{eq110506b} and~\eqref{eq110506c}, we see that
 $$(z+t\xi)_{i+j}(\gamma_{i,s}m_j)=\gamma_{i,s}(z+t\xi)_{j}(m_j)$$
 for all $i$, for $s=1,\ldots,r_i$ and for all $m_j\in M_j$ for each  $j$.
So, Lemma~\ref{lem110408d} implies that $z+t\xi$ is a cycle in $\Hom[A]{M}{M'}_0$.
Since each $z_i$ is bijective
and $t\in \mathfrak m$, Nakayama's Lemma implies that for every $i$, the map $z_i+t\xi_i$ is also bijective.
Hence $z+t\xi$ is an isomorphism $M\xra\cong M'$, so $C\simeq M\cong M'\simeq C'$, as desired.
\end{proof}

Here is Main Theorem~\eqref{item120131b} from the introduction.

\begin{cor}\label{cor110516a}
Let $\underline{t}=t_1,\cdots,t_n$ be a sequence of elements of $R$, and
assume that $R$ is local and $\underline tR$-adically  complete.
Let $D$ be a
DG $K^R(\ul t)$-module that is homologically bounded below and homologically degreewise finite.
If $D$ is quasi-liftable
to $R$ and $\Ext[K^R(\ul t)]{1}DD=0$, then 
any two 
homologically degreewise finite quasi-lifts of $D$ to $R$ are quasiisomorphic
over $R$.
\end{cor}

\begin{proof}
%The existence of a lifting of $D$ to $R$ follows
By induction on $n$, using Theorem \ref{uniqueness} and Proposition~\ref{vanishing}.
\end{proof}

We conclude the paper with an example showing that the quasi-lifts
in the previous two results must be homologically degreewise finite.

\begin{ex}\label{ex120126a}
Let $(R,\m)$ be a local integral domain that is not a field.
Let $Q(R)$ be the field of fractions of $R$, and let $0\neq t\in \m$.
If $F$ is an $R$-free resolution of $Q(R)$,
then $F$  and $0$ are both quasi-lifts of
$0$ from $K^R(t)$ to $R$.
\end{ex}

%\section*{Acknowledgments}

%\bibliography{../+new}
\providecommand{\bysame}{\leavevmode\hbox to3em{\hrulefill}\thinspace}
\providecommand{\MR}{\relax\ifhmode\unskip\space\fi MR }
\providecommand{\MRhref}[2]{%
  \href{http://www.ams.org/mathscinet-getitem?mr=#1}{#2}
}
\providecommand{\href}[2]{#2}

\end{document}